\documentclass[11pt]{amsart}

\usepackage{times}      
\usepackage[colorlinks,citecolor=blue]{hyperref}
\usepackage{latexsym}   
\usepackage{amssymb}    
\usepackage{amsmath}    
\usepackage{amsthm}     
\usepackage{mathrsfs}   
\newcommand{\beq}{\begin{equation}}
\newcommand{\eeq}{\end{equation}}
\newcommand{\bea}{\begin{eqnarray}}
\newcommand{\eea}{\end{eqnarray}}
\newcommand{\beas}{\begin{eqnarray*}}
\newcommand{\eeas}{\end{eqnarray*}}

\newtheorem{theorem}{Theorem}[section]
\newtheorem{definition}[theorem]{Definition}

\newtheorem{proposition}[theorem]{Proposition}
\newtheorem{corollary}[theorem]{Corollary}
\newtheorem{lemma}[theorem]{Lemma}
\newtheorem{remark}[theorem]{Remark}
\newtheorem{example}[theorem]{Example}
\newtheorem{examples}[theorem]{Examples}
\newtheorem{foo}[theorem]{Remarks}

\newcommand{\lp}{\left(}
\newcommand{\rp}{\right)}
\newcommand{\lc}{\left[}
\newcommand{\rc}{\right]}

\newcommand{\hG}{\mathbb G}

\newcommand{\vertiii}[1]{{\left\vert\kern-0.25ex\left\vert\kern-0.25ex\left\vert #1 
    \right\vert\kern-0.25ex\right\vert\kern-0.25ex\right\vert}}

\def\ch{{\mathcal H}}
\def\bch{\bar{\mathcal H}}

\def\cH{{\mathcal H}}

\def\me{{\mathbb  E}}

\def\mr{{\mathbb  R}}

\def\mp{{\mathbb  P}}
\def\D{{\mathbf D}}
\def\J{{\mathbf J}}

\def\bch{{\bar{\mathcal{H}}}}

\def\eps{\varepsilon}

\numberwithin{equation}{section}

\title[]{Density of the signature process of fBm}

\author{Fabrice Baudoin}
\address{Department of Mathematics, University of Connecticut, Storrs, CT 06269.}

\email{fabrice.baudoin@uconn.edu}

\author{Qi Feng}
\address{Department of Mathematics, University of Southern California, Los Angeles, CA 90089.}
\email{qif@usc.edu}

\author{Cheng Ouyang}
\address{Department of Mathematics, Statistics, and Computer Science, University of Illinois at Chicago, Chicago, IL 60607. }
\email{couyang@math.uic.edu}
\thanks{The third author was supported in part  by Simons grant \#355480}

\date{April 19, 2019.}


\begin{document}
\maketitle

\begin{abstract}
We study the density of the signature of fractional Brownian motions with parameter $H>1/4$. In particular, we prove existence, smoothness, global Gaussian upper bounds and Varadhan's type asymptotics for this density. A key result is that the estimates on the density we obtain are controlled by the Carnot-Carath\'eodory distance of the group.
\end{abstract}

\tableofcontents

\section{Introduction}
\noindent{\bf Background and motivation.} For a vector-valued path $\gamma$ with bounded variation, the {\it signature} of $\gamma$ (up to order $N$) is defined from the iterated integrals of $\gamma$. More precisely,
\begin{align}\label{signature intro}S_N(\gamma)_t=\sum_{k=0}^N\int_{0<t_1<\cdots<t_k<t}d\gamma_{t_1}\otimes\cdots\otimes d\gamma_{t_k},\quad t\in[0,1],\end{align}
where we have taken the convention that $S_0(\gamma)_t\equiv 1$. It was first introduced by Chen \cite{Chen54} in the 50's  to study homotopy theory and loop space homology. 

More recent study reveals that the signature $S_N(\cdot)$ can be extended to a much larger class of paths that becomes a fundamental object in Lyons' rough path theory \cite{Lyons Saint-Flour}. In particular, for a $d$-dimensional fractional Brownian motion $B$, it is known that the signature $S_N(B)_t$ of $B$ exists almost surely  when the Hurst parameter $H>1/4$ (\cite[Chapter 15]{FV2010}). Clearly, $S_N(B)_t$ lives in the truncated tensor algebra $T_N(\mr^d)$. By the very definition in \eqref{signature intro},  $S_N(B)_t$ satisfies a canonical SDE on $T_N(\mr^d)$,
\begin{align}\label{signature SDE1 intro}
dS_N(B)_t=S_N(B)_t\otimes dB_t.
\end{align}
Regarding $T_N(\mr^d)$ as a flat linear space, we can recast equation \eqref{signature SDE1 intro} in a more Euclidean way as follows,
\begin{align}\label{signature SDE2 intro}
dS_N(B)_t=\sum_{i=1}^d W_i(S_N(B)_t) dB^i_t.
\end{align}
Here $W_i, i=1,...,d$ are polynomial vector fields on $T_N(\mr^d)$ (see e.g. equation (3.3) of Kusuoka-Stroock \cite{KS87} for an explicit formula of $W_i$'s). We remark here that both equation \eqref{signature SDE1 intro} and \eqref{signature SDE2 intro} are understood in the framework of Lyons' rough path theory.

However, the tensor algebra $T_N(\mr^d)$ is too large for the the process $S_N(B)_t$. The signature $S_N(B)_t$ indeed lives in a strict subspace $\mathbb{G}_N(\mr^d)$ of $T_N(\mr^d)$, known as the {free Carnot group} over $\mr^d$ of step $N$ (see, e.g., \cite{Bau, BC06}). Therefore we can restrict equation \eqref{signature SDE2 intro} to $\mathbb{G}_N(\mr^d)$, and in this case the vector fields $W_i$ form a uniform hypoelliptic system  on $\mathbb{G}_N(\mr^d)$. We thus are  interested in the existence of a probability density function of $S_N(B)_t$ with respect to the Haar measure on $\mathbb{G}_N(\mr^d)$ and in the properties of this density, if it exists.

\bigskip
\noindent{\bf Main results.} Hypoeliptic SDEs driven by a fractional Brownian motion have been studied extensively in recent literature. For example, the existence of a smooth density function has been proved in \cite{BH, CHLT},  Varadhan estimates has been established in \cite{BOZ15}, and complete small time asymptotics is obtained in a recent preprint \cite{IN19}.  However, all the aforementioned works assume that the vector fields are $C^\infty$-bounded. In our current situation, the vector fields $W_i$'s are  of polynomial order, hence not bounded.  

Despite the  technical difficulty mentioned above regarding unbounded vector fields, the main motivation of our investigation is: (1) the signature  $S_N(B)_t$ of $B$ is a canonical process on the Lie group $\mathbb{G}_N(\mr^d)$ satisfying a canonical hypoelliptic SDE. A better understanding of it may shed a light in understanding more general hypoellitic SDEs; (2) we are interested to see whether the group and dilation structure on $\mathbb{G}_N(\mr^d)$ can help us obtaining sharper results than in a general setting. 

The main results obtained through our investigation are summarized as follows.
\begin{itemize}
\item[(i)] When $H>1/4$, the signature of the fractional Brownian motion $X_t=S_N(B)_t$ admits a smooth density function on $\mathbb{G}_N(\mr^d)$ with respect to the Haar measure of $\mathbb{G}_N(\mr^d)$.
\item[(ii)] Denote by $p_t(g)$ the density function of $X_t$ in (i), we have
\begin{align*}
p_t(g)\leq \frac{C}{t^{\nu/2}} e^{-\frac{\|g\|_{\textsc{CC}}^2}{Ct^{2H}}}, \quad\mathrm{for\ all}\ g\in\frak{g}_N(\mr^d).
\end{align*}
In the above, $\nu$ is the Hausdorff dimension (=homogeneous dimension) of $\mathbb{G}_N(\mr^d)$ and $\|\cdot\|_{\textsc{CC}}$ is the Carnot-Carath\'{e}odory norm on $\mathbb{G}_N(\mr^d)$, that is, the Carnot-Carath\'{e}odory distance between $g$ and $\mathbf{1}$, the group identity of $\mathbb{G}_N(\mr^d)$.
\item[(iii)] The density $p_t(g)$ is strictly positive on $\mathbb{G}_N(\mr^d)$. Hence, as an easy corollary of the self-similarity of $B$, one has
$$p_t(g)\geq\frac{c}{t^{\nu/2}},$$
for all $g\in\mathbb{G}_N(\mr^d)$ with $\|g\|_{\textsc{CC}}\leq t^H$.
\item[(iv)] Let $p_\epsilon(g)$ be the density of $S_N(\epsilon B)_t$. The following Varadhan estimate holds
\begin{align*}
\liminf_{\eps\downarrow 0} \eps^2\log p_\eps(g)\geq -\frac{1}{2}d^2_R(g),\quad
\textrm{and}\quad
\limsup_{\eps\downarrow0}\eps^2\log p_\eps(g)\leq -\frac{1}{2}d^2(g).
\end{align*}
Here $d$ and $d_R$ are the controlling ``distances" associated to equation \eqref{signature SDE2 intro} (see \eqref{def: d1} and \eqref{def: dr1} for a precise definition). Moreover, both $d$ and $d_R$ are equivalent to the Carnot-Carath\'eodory distance on $\mathbb{G}_N(\mr^d)$.
\end{itemize}
  
 Several remarks are in order.
 
 \begin{remark}
 The existence of a smooth density for hypoelliptic SDEs driven by fractional Brownian motions has been established in \cite{BH, CHLT} and (inexplicitly) in \cite{BOZ15}. But all of the aforementioned  works assumed  $C^\infty$-bounded vector fields. So some extra effort is needed in our present setting since we are dealing  with unbounded vector fields.
 \end{remark}

 \begin{remark}
 It is generally expected that under suitable boundedness and non-degeneracy condition on the vector fields, SDEs driven by a fractional Brownian motion admit a Gaussian type density upper bound. But due to the limitation of rough path estimate and hence a lack of Gaussian concentration for the solution, the best results in the literature in this line are only sub-Gaussian bounds (see, e.g., \cite{BOT14, GOT18}).  Note that the upper bound in (ii) is of Gaussian type, and it  seems to be a first positive answer in this regard in a non-trivial setting. We also would like to mention that the power of $t$ before the exponential in (ii) is sharp and matches the local lower bound in (iii).
 \end{remark}

 \begin{remark}
 The strict positivity of the density of the signature was proved in the Brownian case in Kusuoka-Stroock \cite{KS87}, using support theorem and Markov property. This approach certainly breaks down in our current setting, as fractional Brownian motions are in general not Markovian. To overcome this difficulty, we have to resort to an approach based on Malliavin calculus (see, e.g., \cite{B-N Saint-Flour}). The proof then boils down to verifying that the It\^{o}-Lyons map associated to equation \eqref{signature SDE2 intro} is a submersion from the Cameron-Martin space of $B$ to $\mathbb{G}_N(\mr^d)$ (see Theorem \ref{immersion} below).
 \end{remark}
 
 \begin{remark}By comparing the exponential terms in (ii) and (iii) to the Varahdan estimate in (iv), it is natural to wonder whether the controlling ``distances" $d$ and $d_R$ are comparable to the Carnot-Carath\'{e}odory distance $\|\cdot\|_{\textsc{CC}}$.  We are able to give it an affirmative answer.  This is considered the main contributions in the section of Varadhan estimate of this paper. 
\end{remark}

\begin{remark}
 We would like to say a few more words about the controlling ``distance". Indeed, we do not know whether it is  a distance function -- it seems difficult to establish the triangle inequality. Moreover,  it is not even clear to us that they are continuous, though the lower semi-continuity always holds, being  good rate functions in certain large deviation principles.
\end{remark}

\begin{remark}
Finally, we would like to mention that for the sake of presentation, we restricted ourselves to the case of the free Carnot group $\mathbb{G}_N(\mr^d)$, however the same results apply without any change to the study of the equation
\[
dX_t=\sum_{i=1}^d W_i(X_t) dB^i_t
\]
where the $W_i$'s form a basis of the first layer of any Carnot Lie algebra (not necessarily free) and $B$ is a fractional Brownian motion with parameter $H>1/4$.
\end{remark}

\bigskip
The rest of the paper is organized as follows. In Section 2, we present some preliminary materials on Malliavin calculus and free Carnot groups. Section 3 is devoted to the existence of the density function $p_t$. The upper bound of $p_t$ is then derived in Section 4. In Section 5, we prove the strict positivity of $p_t$, and in the last section we establish the Varadhan estimate and the equivalence of distances.

\section{Preliminaries}

\subsection{Malliavin calculus} \label{intro Malliavin calculus}

To fix notations and the conventions we use, we introduce the basic framework of Malliavin calculus in this subsection.  The reader is invited to read the corresponding chapters in  \cite{Nualart06} for further details. 

A fractional Brownian motion with parameter $H \in (0,1]$ is a continuous centered Gaussian process with covariance function
\[
R(s,t)=\frac{1}{2} \left( s^{2H}+t^{2H} +|t-s|^{2H} \right).
\]
By a $d$-dimensional fractional Brownian motion $(B_t)_{t \ge 0}$ we mean that $B_t$ can be written
\[
B_t=(B^1_t,\cdots,B^d_t)
\]
where the $B_i$'s are independent fractional Brownian motions with parameter $H$.

Let $\mathcal{E}$ be the space of $\mathbb{R}^d$-valued step
functions on $[0,1]$, and $\mathcal{H}$  the closure of
$\mathcal{E}$ for the scalar product:
\[
\langle (\mathbf{1}_{[0,t_1]} , \cdots ,
\mathbf{1}_{[0,t_m]}),(\mathbf{1}_{[0,s_1]} , \cdots ,
\mathbf{1}_{[0,s_m]}) \rangle_{\mathcal{H}}=\sum_{i=1}^d
R(t_i,s_i).
\]
 $\ch$ is called the reproducing kernel Hilbert space for fractional Brownian motion $B$. If we denote by $e_i, i=1,\dots, d$, the canonical basis of $\mr^d$, one can  construct an isometry $K^*_{H}:\ch\rightarrow L^2([0,1])$ with $$ K_{H}^*\mathbf{1}_{[0,t]}e_i=\mathbf{1}_{[0,t]}K_{H}(t,\cdot)e_i,$$ for some kernel function $K_H$. Furthermore, 
 denote by $\bch([0,1])$  the Cameron Martin space associated with fractional Brownian motion $B$. The kernel $K_H$ give an isometry $K_H: L^2([0,1])\to\bch([0,1])$ by \begin{equation*}
 K_H\psi := \int_0^\cdot K_H(\cdot,s) \psi(s)\, ds.
\end{equation*}
Clearly, by the definition of $K_H^*$ and $K_H$,  $\mathcal{K}=(K_H\circ K_H^*)^{-1}$ is then an isometry from $\bch{([0,1])}$ to $\ch([0,1])$.


Standard isometry arguments allow to define the Wiener integral $B(h)=\int_0^{1} \langle h_s, dB_s \rangle$ for any element $h\in\ch$, with the additional property $\mathbb{E}[B(h_1)B(h_2)]=\langle h_1,\, h_2\rangle_{\ch}$ for any $h_1,h_2\in\ch$.
  An $\mathcal{F}$-measurable real
valued random variable $F$ is said to be cylindrical if it can be
written, for a given $n\ge 1$, as
\begin{equation*}
F=f\lp  B(\phi^1),\ldots,B(\phi^n)\rp,
\end{equation*}
where $\phi^i \in \mathcal{H}$ and $f:\mathbb{R}^n \rightarrow
\mathbb{R}$ is a $C^{\infty}$ bounded function with bounded derivatives. The set of
cylindrical random variables is denoted by $\mathcal{S}$.

The Malliavin derivative is defined as follows: for $F \in \mathcal{S}$, the derivative of $F$ is the $\mathbb{R}^m$ valued
stochastic process $(\mathbf{D}_t F )_{0 \leq t \leq 1}$ given by
\[
\mathbf{D}_t F=\sum_{i=1}^{n} \phi^i (t) \frac{\partial f}{\partial
x_i} \left( B(\phi^1),\ldots,B(\phi^n)  \right).
\]
More generally, we can introduce iterated derivatives by
$
\mathbf{D}^k_{t_1,\ldots,t_k} F = \mathbf{D}_{t_1}
\ldots\mathbf{D}_{t_k} F.
$
\ For any $p \geq 1$,  we denote by
$\mathbb{D}^{k,p}$ the closure of the class of
cylindrical random variables with respect to the norm
\[
\left\| F\right\| _{k,p}=\left( \mathbb{E}\left( F^{p}\right)
+\sum_{j=1}^k \mathbb{E}\left( \left\| \mathbf{D}^j F\right\|
_{\mathcal{H}^{\otimes j}}^{p}\right) \right) ^{\frac{1}{p}},
\]
and
\[
\mathbb{D}^{\infty}=\bigcap_{p \geq 1} \bigcap_{k
\geq 1} \mathbb{D}^{k,p}.
\]

Let $F=(F^1,\ldots , F^n)$ be a random vector whose components are in $\mathbb{D}^\infty$. Define the Malliavin matrix of $F$ by
$$\gamma_F=(\langle \mathbf{D}F^i, \mathbf{D}F^j\rangle_{\ch})_{1\leq i,j\leq n}.$$
Then $F$ is called  {\it non-degenerate} if $\gamma_F$ is invertible $a.s.$ and
$$(\det \gamma_F)^{-1}\in \cap_{p\geq1}L^p(\Omega).$$
\noindent
It is a classical result that the law of a non-degenerate random vector admits a smooth density with respect to the Lebesgue measure on $\mr^n$.

\subsection{Signature, Log-signature and Free Carnot groups. }


The truncated tensor algebra $T_N(\mathbb{R}^{d})$ over $\mathbb{R}^{d}$ is given by
\[
T_N(\mathbb{R}^{d})=\bigoplus^{N}_{k=0}(\mathbb{R}^{d})^{\otimes k}
\]
with the convention that $(\mathbb{R}^{d})^{0}=\mathbb{R}$. All computations in this truncated algebra are done at degree at most $N$, i.e. $e_{i_1} \otimes  \cdots \otimes e_{i_k} =0$, if $k \ge N$.

\begin{definition}[Signature of the fractional Brownian motion]
Let $(B_t)_{t \ge 0}$ be a fractional Brownian motion with parameter $H>1/4$. The $T_N(\mathbb{R}^{d})$-valued path
\[
 X_t=S_N(B)_t=\sum_{k=0}^{N} \int_{\Delta^k [0,t]}  dB^{\otimes k}, \quad t \ge 0,
\]
is called the signature of $B$ of order $N$. 
\end{definition}
The iterated integrals appearing in the definition of the signature are understood in the sense of rough paths.
Note that the signature is the solution to a rough differential equation that writes
\begin{align}\label{SDE: signature}
	dX_t=X_t\otimes dB_t=\sum_{i=1}^dW_i(X_t)dB_t^i,\quad X_0=\mathbf{1}=(1,0,0,...)\in T_N(\mr^d),
\end{align}
where the $W_i$'s are polynomial vector fields on $T_N(\mr^d)$.

 It is a well-known theorem by Chen \cite{Chen54}  that the signature of a path is a Lie element. To be more precise, consider
\[
\hG_N(\mr^d) =\exp ( \mathfrak{g}_N(\mathbb{R}^d)),
\]
where $\mathfrak{g}_N(\mathbb{R}^d)$ is the Lie sub-algebra of $T_N(\mathbb{R}^{d})$ generated by the canonical basis $e_i, i=1,\dots,d,$   of $\mathbb{R}^d$, and the Lie bracket  is given by $[a,b]=a\otimes b-b\otimes a$.  Then Chen's theorem (see \cite{BC06} for the Lie group rough path version) asserts that for every $t \ge 0$, we almost surely have $X_t \in \hG_N(\mr^d) $.

Indeed, from the Chen-Strichartz formula (see \cite{Bau}), one has the following explicit following representation of $X_t$:
\[
X_t= \exp \left( \sum_{I, l(I) \le N } \Lambda_I (B)_t e_{I} \right),
\]
where:
\begin{itemize}
\item For $z \in T_N(\mathbb{R}^{d})$, $\exp (z) =\sum_{k=0}^N \frac{1}{k!} z^{\otimes k}$;
\item If $I\in \{1,...,d\}^k$ is a word, 
\[
e_I= [e_{i_1},[e_{i_2},...,[e_{i_{k-1}},
e_{i_{k}}]...],
\]
and
$
l(I)=k;
$
\item Let $\mathcal{S}_k$ be the set of the permutations of
$\{1,...,k\}$, then
\[
\Lambda_I (B)_t= \sum_{\sigma \in \mathcal{S}_k} \frac{\left(
-1\right) ^{e(\sigma )}}{k^{2}\left(
\begin{array}{l}
k-1 \\
e(\sigma )
\end{array}
\right) } \int_{0 \leq t_1 \leq ... \leq t_k \leq t} 
dB^{\sigma^{-1}(i_1)}_{t_1}  \cdots  
dB^{\sigma^{-1}(i_k)}_{t_k},\quad t \ge 0.
\]
In the above, for $\sigma \in \mathcal{S}_k$,  $e(\sigma)$ is the cardinality of the set
$
\{ j \in \{1,...,k-1 \} , \sigma (j) > \sigma(j+1) \}.
$
\end{itemize}

One should not expect a probability density function of $X_t$ with respect to the Lebesgue measure of the flat space $T_N(\mr^d)$.  Instead, one should expect a density of $X_t$ with respect to the Haar measure of $\hG_N(\mr^d)$. In order to prove the existence of such a density function, by reducing to local coordinate chats, it suffices to show that there exists a random variable $M$ with negative moment to any order such that
\begin{align}\label{existence 1}u^*\gamma_t(X)u\geq M||u||^2\end{align}
for all $u\in T_{X_t}(\hG_N(\mr^d)$, the tangent space of $\hG_N(\mr^d)$ at $X_t$.  Here
$$\gamma_{t}^{ij}(X)=\langle\D X_t^i, \D X_t^j\rangle_{\mathcal{H}}.$$
Although  (\ref{existence 1}) can be proved directly, it is sometimes more convenient to work on a flat space than a curved space when using Malliavin calculus. Hence we will  take another route in order to prove the existence of the density of $X_t$ and introduce the log-signature of the fractional Brownian motion.

It is well known that $\hG_N(\mr^d)$ is a free nilpotent  and simply connected Lie group whose Lie algebra inherits, from the grading of $T_N(\mathbb{R}^{d})$, a stratification
\[
\mathfrak{g}_N(\mathbb{R}^d)=\mathcal{V}_{1}\oplus \cdots \oplus \mathcal{V}_{N},
\]
with
\begin{equation*}
\dim \mathcal{V}_{j}= \frac{1}{j} \sum_{i \mid j} \mu (i)
d^{\frac{j}{i}}, \text{ } j \leq N,
\end{equation*}
 where $\mu$ is the M\"obius function. From the Hall-Witt theorem we can then construct a basis of $\mathfrak{g}_N(\mathbb{R}^d)$ which is adapted to the stratification
\[
\mathfrak{g}_N(\mathbb{R}^d)=\mathcal{V}_{1}\oplus \cdots \oplus \mathcal{V}_{N},
\]
and such that every element of this basis is an iterated bracket of the $e_i$'s.  Let $\mathcal{B}$ denote such a basis and for $x \in \mathfrak{g}_N(\mathbb{R}^d)$, let $[x]_\mathcal{B}\in \mathbb{R}^n$ be the coordinate vector of $x$ in the basis $\mathcal{B}$ where $n=\dim \mathfrak{g}_N(\mathbb{R}^d)$.

\begin{definition}
[Log-signature of the fractional Brownian motion]
Let $(B_t)_{t \ge 0}$ be a fractional Brownian motion with parameter $H>1/4$. The log-signature of order $N$ of $(B_t)_{t \ge 0}$ is the $\mathbb{R}^n$-valued process
\[
U_t=[\exp^{-1}( X_t)]_\mathcal{B} = \sum_{I, l(I) \le N } \Lambda_I (B)_t [e_{I}]_\mathcal{B}.
\] 
\end{definition}

From \cite{Bau}, $(U_t)_{t\ge0}$ solves in $\mathbb{R}^n$ a rough differential equation

\begin{align}\label{SDE u}U_t=\sum_{i=1}^d\int_0^tV_i(U_s)dB^i_s,\end{align}
where the vector fields $V_1,\cdots,V_d$ are polynomial and generate a Lie algebra isomorphic to $ \mathfrak{g}_N(\mathbb{R}^d)$. Actually $V_1,\cdots,V_d$ are left invariant vector fields for a polynomial group law $\star$ on $\mathbb{R}^n$ such that $(\mathbb{R}^n,\star)$ is isomorphic to $\mathbb{G}_N(\mathbb{R}^d)$. The stratification of $\mathbb{R}^n$ induced by this group law will be written

\begin{align}\label{strati}
\mathbb{R}^n=\mathcal{U}_{1}\oplus \cdots \oplus \mathcal{U}_{N}.
\end{align}
Since the exponential map $\mathfrak{g}_N(\mathbb{R}^d) \to  \hG_N(\mr^d)$ is a diffeomorphism, and because the Haar measure of $\hG_N(\mr^d)$ is induced  by the Lebesgue measure of $\mathfrak{g}_N(\mr^d)$ (through the exponential map), the existence of a smooth density of $U_t$ can easily translate to that of $X_t$. Hence, in what follows, we focus on proving the existence of a smooth density for $U_t$.

\begin{remark}
Note that components of $X_t$ are iterated integrals of $B$ to a certain order, hence $\|X_t\|$ has finite moments to any order.  Similarly $\|U_t\|$ also has finite moments to any order. 
\end{remark}

We end this subsection with a
global scaling property for the signature $X_t$ and the log-signature $U_t$. On $\mathfrak{g}_N(\mathbb R^d)$ we can
consider the family of linear operators $\delta_{\lambda}:\mathfrak{g}_N(\mathbb R^d)
\rightarrow \mathfrak{g}_N(\mathbb R^d)$, $\lambda \geq 0$ which act by scalar
multiplication $\lambda^{i}$ on $\mathcal{V}_{i} $. These operators are
Lie algebra automorphisms due to the grading. The maps $\delta_{\lambda}$
induce Lie group automorphisms $\Delta_{\lambda} :\mathbb{G}_N(\mathbb R^d) \rightarrow
\mathbb{G}_N(\mathbb R^d)$ which are called the canonical dilations of
$\mathbb{G}_N(\mathbb R^d)$. With an abuse of notation, we also denote $\Delta_\lambda$ the induced non-homogeneous dilation on $\mathbb{R}^n$, i.e. for $u \in \mathfrak{g}_N(\mathbb R^d)$,  $\Delta_\lambda [ u ]_\mathcal{B}=[\delta_{\lambda} u ]_\mathcal{B}$.

\begin{proposition}\cite[Theorem 3.4]{BC06}
Let $(\Delta_\lambda)_{\lambda \ge 0}$ be the one parameter family of dilations on $\mathbb{G}_N(\mathbb{R}^d)$. Then,
\begin{align}\label{global scaling}
	(X_{ct})_{t \geq 0} =^{law} (\Delta_{c^H} X_{t})_{t \geq 0}.
\end{align}
and 
\begin{align}\label{global scaling 2}
	(U_{ct})_{t \geq 0} =^{law} (\Delta_{c^H} U_{t})_{t \geq 0}.
\end{align}
\end{proposition}

\subsection{ Carnot-Carath\'eodory distance} 

There is a canonical sub-Riemannian distance on any  Carnot group.
Let $\mathbb{G}$ be a Carnot group whose Lie algebra is stratified as
 \[
\mathfrak{g}=\mathcal{U}_{1}\oplus \cdots \oplus \mathcal{U}_{N}.
\]
and assume that $\mathcal{U}_{1}$ is equipped with an inner product. Using left invariance, the first layer $\mathcal{U}_{1}$ induces a left-invariant bracket generating sub-bundle (still denoted $\mathcal{U}_{1}$) in the tangent bundle of $\mathbb{G}$. Left-invariance also allows to define, from the inner product, a left invariant sub-Riemannian metric on $\mathcal{U}_{1}$ as follows. A $C^1$-curve $\gamma : [0,1] \to \mathbb{G} $ is called \textit{horizontal} if for every $t \in [0,1]$, $\gamma'(t) \in \mathcal{U}_{1}$. For $g_1,g_2 \in \mathbb{G}$, one defines the Carnot-Carath\'eodory distance as
\[
d(g_1,g_2)=\inf_{\mathcal{S}(g_1,g_2)} \int_0^1 \| \gamma'(t) \| dt,
\]
where $\mathcal{S}(g_1,g_2)$ is the set of $C^1$ horizontal curves $\gamma$ such that $\gamma(0)=g_1$, $\gamma(1)=g_2$.
For later use, we record the following well-known  properties of $d$:

\begin{proposition}
\

\begin{itemize}
\item For $g_1,g_2 \in \mathbb{G}$,
\[
d(g_1,g_2)=d(g_2,g_1)=d(0,g_1^{-1} g_2).
\]
\item Let $(\Delta_\lambda)_{\lambda \ge 0}$ be the one parameter family of dilations on $\mathbb{G}$. For $g_1,g_2 \in \mathbb{G}$, and $\lambda \ge 0$,
\[
d(\Delta_{\lambda} g_1,\Delta_{\lambda} g_2)=\lambda d(g_1,g_2).
\]
\end{itemize}
\end{proposition}
%
%
The Carnot-Carath\'eodory distance is pretty difficult to explicitly compute in general. It is often much more convenient to estimate using  homogeneous norms.

\begin{definition}
 A homogeneous norm on $\mathbb{G}$ is a continuous function $\parallel \cdot \parallel : \mathbb{G} \rightarrow
[0,+\infty) $, such that:
\begin{enumerate}
\item $\parallel \Delta_{\lambda} g \parallel=\lambda \parallel g \parallel$, $\lambda \ge 0 $, $g \in \mathbb{G}$;
\item $\parallel g \parallel=0$ if and only if $g=\bf{1}$, the group identity of $\mathbb{G}$.
\end{enumerate}
\end{definition}


It turns out that the Carnot-Carath\'eodory distance  is equivalent to any homogeneous norm in the following sense:

\begin{theorem}
Let $\parallel \cdot \parallel $ be a homogeneous norm on $\mathbb{G}$. There exist two positive constants $C_1$ and $C_2$ such that for every $g_1,g_2 \in \mathbb{G}$,
\[
A \| g_1^{-1}g_2 \| \le d(g_1,g_2) \le B \|g_1^{-1}g_2 \|.
\]
\end{theorem}

If $\mathbb{G}$ is a Carnot group with Carnot-Carath\'eodory distance $d$ and identity element $\mathbf{1}$, the C-C norm of $g \in \mathbb{G}$ will be defined by $\| g \|_{CC}=d(\mathbf{1},g)$.

\section{Existence of density}

In this section, we prove that  the log-signature $U_t$, $t>0$ has a smooth density with respect to the Lebesgue measure. Consider in $\mathbb{R}^n$ the system

\begin{align}\label{SDE u2}U^x_t=x+\sum_{i=1}^d\int_0^tV_i(U_s)dB^i_s.\end{align}
In particular, $U^0_t=U_t$. 

\begin{lemma}\label{eigenvalue integrability general}
Let $C$ be a symmetric nonnegative definite $m\times m$ random  matrix.  Assume that the entries $C_{i,j}$ have moments of all orders. Then the largest eigenvalue of $C$ has moments of all orders.
\end{lemma}
\begin{proof}
Let $\lambda$ be the largest eigenvalue of $C$ and set $|C|=(\sum_{i,j=1}^m C_{i,j}^2)^{\frac{1}{2}}$. we have
\begin{align}
\mathbb{P}\{\lambda>x\}&=\mathbb{P}\{\sup_{|v|=1} v^TCv> x\}\nonumber\\
&=\mathbb{P}\{\sup_{|v|=1} v^TCv> x, |C|\leq x/8\}+\mathbb{P}\{|C|>x/8\}.\label{probability lambda}
\end{align}
Fix any $v_0$ with $|v_0|=1$, we have $|v^TCv-v_0^TCv_0|\leq 2|C||v-v_0|$.  Hence when $|C|<x/8$,
\begin{align*}
v_0^TCv_0\geq v^TCv-2|C||v-v_0|\geq x-\frac{x}{2}=\frac{x}{2}.
\end{align*}
Thus the probability in (\ref{probability lambda}) is bounded by
\begin{align*}
\mathbb{P}\{v_0^TCv_0\geq x/2\}+\mathbb{P}\{|C|>x/8\}\leq \frac{2^p \me|C|^p}{x^P}+\frac{8^p\me|C|^p}{x^p}.
\end{align*}
Clearly this implies that $\lambda$ has finite moments to any order.
\end{proof}

\begin{proposition}\label{integrality JJ}
Consider the $n \times n$ Jacobian matrix $\mathbf{J}_t=\frac{\partial U_t^x}{\partial x} $. The largest eigenvalue of $$(\J^*_t \J_t)^{-1}$$ has finite moments to any order.
\end{proposition}
\begin{proof}
From Lemma \ref{eigenvalue integrability general}, it is enough to prove that entries of both $\J_t$ and $\J_t^{-1}$ have finite moments to any order.
The stochastic differential equation

\begin{align*}U^x_t=x+\sum_{i=1}^d\int_0^tV_i(U_s)dB^i_s\end{align*}
can be integrated as
\[
U^x_t = x \star U_t
\]
where $\star$ is the polynomial group law on $\mathbb{R}^n$ introduced before  such that $(\mathbb{R}^n,\star)$ is isomorphic to $\mathbb{G}_N(\mathbb{R}^d)$. Since the inverse of the map $ x \to  x \star U_t$ is clearly $ x \to x \star (U_t)^{-1}$ the conclusion follows from the fact that both $U_t$ and $U_t^{-1}$ have finite moments to any order since they are linear combinations of iterated integrals of the fractional Brownian motion.
\end{proof}

Now we are ready to state and prove our main result in this section.
\begin{theorem}\label{integrability small e-v}
Let $\gamma_t(U^x)=\langle \D U_t^x, \D U_t^x\rangle_\mathcal{H}$ be the Malliavin matrix of $U_t^x$, and $\lambda_t$ the smallest eigenvalue of $\gamma_t(U^x)$. Then
\begin{align}\label{M matrix of u}\sup_{t\in(0,1]}\left\|\frac{t^{2HN}}{\lambda_t}\right\|_p<\infty,\end{align}
for all $p\geq 1$. In particular, $U_t^x$ admits a smooth density with respect to the Lebesgue measure of $\mathbb{R}^n$.
\end{theorem}

\begin{proof} 
If the vector fields $V_i$'s were $C^\infty$-bounded, Lemma 3.9 in \cite{BOZ15} can be easily translated to the estimate claimed here, due to the self-similarity of the fractional Brownian motion.
In what follows, we show that the conclusion of the theorem is still true for  $U_t^x$, even though the vector fields $V_i$'s are not bounded ($V_i$'s are of polynomial growth).  

For $\epsilon\in(0,1]$ and $V_i^\epsilon=\epsilon V_i$, consider the following family of SDEs,
$$dU^{x,\epsilon}_t=\sum_{i=1}^dV_i^\epsilon(U^{x,\epsilon}_t)dB_t^i, \quad U^{x,\epsilon}_0=x\in \mathbb{R}^n.$$ 
Let $\J^{\epsilon}_t$ be the Jacobian of $U^{x,\epsilon}_t$ and $\beta_I^{J,\epsilon}(t,x)$ be such that 
$$(\mathbf{J}_t^\epsilon)^{-1}V^\epsilon_{[I]}(U^{x,\epsilon}_t)=\sum_{J\in\mathcal{A}_1(N)}\beta_{I}^{J,\epsilon}(t,x)V^\epsilon_{[J]}(x).$$
Then for any $I,J\in\mathcal{A}_1(N)$,  let $\beta^{I,\epsilon}(\cdot,x)$ be the column vector $(\beta^{I}_i(\cdot, x))_{i=1,...,d}$ and defined
$$M_{I,J}^\epsilon(x)=\langle\beta^{I}(\cdot,x), \beta^{J}(\cdot, x)\rangle_{\mathcal{H}},$$
which is considered here as a (symmetric) matrix indexed by $I, J\in\mathcal{A}_1(N)$.
Denote by \begin{align*}
\gamma_t(U^{x,\epsilon}))&=\langle\D U^{x,\epsilon}_t, \D U^{x,\epsilon}_t\rangle_{\mathcal{H}}
\end{align*}
the Malliavin matrix of $U^{x,\epsilon}_t$.  It has been  shown in \cite{BOZ15} that for come constant $C$ not depending on $\epsilon$,
$$\lambda_{\min}(\gamma_1(U^{x,\epsilon}))\geq C\epsilon^{2N}\lambda_{\min}(M^\epsilon_{I,J}(x))\lambda_{\min}((\J^\epsilon_1)^*\J^\epsilon_1).$$
Here  $\lambda_{\min}$ stands for the smallest eigenvalue of the corresponding matrices. Hence the claimed result relies on the integrability of $\lambda^{-1}_{\min}(M^\epsilon_{I,J}(x))$ and $\lambda^{-1}_{\min}((\J^\epsilon_1)^*\J^\epsilon_1).$  The integrability of $\lambda^{-1}_{\min}((\J^\epsilon_1)^*\J^\epsilon_1)$ has been taken care of in Proposition \ref{integrality JJ}.  In what follows, we justify that $\lambda^{-1}_{\min}(M^\epsilon_{I,J}(x))$ also has finite moments to any order even though the vector fields $V_i$'s are not bounded.

Note that  Assumption 3.1 is assumed in \cite{BOZ16} in order to prove that the largest eigenvalue of $(M^\epsilon_{I,J})^{-1}$ has finite moments to any order (uniformly in $\epsilon\in[0,1]$). The key in the assumption is that one can find functions $\omega^J_I$ such that $V_{I}(x)=\sum_{J\in\mathcal{A}(l)}\omega_I^J(x)V_{[J]}(x)$, and most importantly $\omega^J_I$'s are $C^\infty$-bounded. Indeed, all the later argument in \cite{BOZ16} are based on the fact that $\omega^J_I$ are $C^\infty$-bounded and that $\beta_J^{I,\epsilon}(t,x)$ are defined by SDEs that only involve $\omega^J_I$'s (see \cite[Equation 3.5]{BOZ16}).  It does not rely on the boundedness of the vector fields $V_i$'s themselves.  

In our present situation, since $V_i$'s are nilpotent, we can simply take
$\omega_I^J=\delta^J_I,$ if $|I|\leq N$; and $\omega_I^J=0$ if $|I|>N$. For this particular choice of $\omega^J_I$, they are clearly $C^\infty$-bounded. Hence we can conclude that in our present situation, the largest eigenvalue of $(M_{I,J}^\epsilon)^{-1}$ still has finite moments to any order and uniform in $\epsilon\in[0,1]$. The proof is thus completed.
\end{proof}

\section{Upper bound of the density}



This section is devoted to prove a sharp upper bound for the density $p_t(u)$ of $U_t\in \mathbb{R}^n$.    
In order to establish an upper bound for $p_t(u)$, we aim to use the following general bound borrowed from inequality (24) of \cite{BNOT16} (also see inequality (21) of \cite{BOT14}).
\begin{align}\label{formula density}
p_t(u)\le c\mathbb P(U_t\geq u)^{1/2}\|\gamma_t(U)^{-1}\|^{n}_{n,2^{n+2}}\|\D U_t\|^n_{n,2^{n+2}}\end{align}
In the above,  $U_t\ge u$ means that the inequality holds component-wise. Without loss of generality, we may assume $u^i\geq 0$ for $1\leq i\leq n$. [An argument explaining why we can assume all the coordinate of $u$ are positive can be found in the proof of Theorem 3.13 in \cite{BOT14}.]

According to the general bound in \eqref{formula density}, in order to obtain an upper bound for the density of $U_t$, we need to estimate (1) the tail probability of $U_t$,  (2) the Malliavin derivatives of $U_t$, and (3) the Malliavin matrix of $U_t$.

\medskip
\noindent{\bf Controlling the tail probability.}\ \  In order to better characterize the event $\{U_t\geq u\}$, let us introduce an homogeneous norm on $\mathbb{R}^n$, namely, set
\begin{align*}
\vertiii{u}\triangleq\max_{i=1,...,n}|u^i|^{q_i},
\end{align*}
where $q_i=1/k$ if $u_i\in\mathcal{U}_k$. By the equivalence of homogeneous norms on $\mathbb{R}^n$, there exists constants $C_1, C_2>0$ such that 
\begin{align}\label{equivalence of homo norm on g}
C_1\|u\|_{\textsc{CC}}\leq  \vertiii{u}\leq C_2\|u\|_{\textsc{CC}},\quad \mathrm{for\ all}\ u\in\mathbb{R}^n.
\end{align}
Now we are ready to state and prove the following lemma concerning the tail probability of $U_t$.
\begin{lemma}\label{lem: tail estimate}There exists a constant $C>0$ such that 
	\begin{equation}\label{tail estimate}
		\mp(U_t\ge u)\leq Ce^{-\frac{\|u\|^2_{\textsc{CC}}}{Ct^{2H}}},
	\end{equation}
	for all $u\in\mathbb{R}^n$.
\end{lemma} 
\begin{proof}
Recall that $\triangle_\lambda$ is the dilation operator on $\mathbb{R}^n$ weighted with respect to the decomposition
\[
\mathbb{R}^n=\mathcal{U}_1 \oplus \cdots \oplus \mathcal{U}_N.
\]
 Since we assumed $u^i\geq0$ for $1\leq i\leq n$, clearly, we have
\begin{align*}
\{U_t\geq u\}=\{\triangle_{\|u\|^{-1}_{\textsc{CC}}}U_t\geq \triangle_{\|u\|^{-1}_{\textsc{CC}}}u\}
\subset\left\{\vertiii{\triangle_{\|u\|^{-1}_{\textsc{CC}}}U_t}\geq \vertiii{\triangle_{\|u\|^{-1}_{\textsc{CC}}} u}\right\}.
\end{align*}
Note that $\|\triangle_{\|u\|^{-1}_{\textsc{CC}}}u\|_{\textsc{CC}}=1$. Then by \eqref{equivalence of homo norm on g} and the fact that $\vertiii{\cdot}$ are homogeneous with respect to the dilation, we obtain
\begin{align*}
\{U_t\geq u\}&\subset\left\{    \vertiii{\triangle_{\|u\|^{-1}_{\textsc{CC}}}U_t}\geq C_1\right\}\\
&=\left\{\|u\|^{-1}_{\textsc{CC}}\,\vertiii{U_t}\geq C_1\right\}\\
&\subset\left\{\|U_t\|_{\textsc{CC}}\geq C_1/C_2\|u\|_{\textsc{CC}}\right\}.
\end{align*}
Therefore
\begin{align*}
\mp(U_t\ge u)&\leq\mp(\|U_t\|_{\textsc{CC}}\ge C_1/C_2\|u\|_{\textsc{CC}})\\
 & \leq  \mp(\|U_1\|_{\textsc{CC}}\ge t^{-H}C_1/C_2\|u\|_{\textsc{CC}}) \\
 & \leq Ce^{-\frac{\|u\|^2_{CC}}{Ct^{2H}}}.
\end{align*}
In the above, we used the fact that from \cite{FO2010}, $\|U_1\|_{\textsc{CC}}$ has Gaussian tail for the last inequality.
\end{proof}

\medskip
\noindent{\bf Estimate of Malliavin derivatives.}\ \  We aim to estimate the Malliavin derivative of $U_t$. In order to handle the $\bch([0,1])^{\otimes k}$-norm of $\mathbf{D}^kU_k$, we use an idea of Inahama \cite{Inahama} which we describe now. More details can be found in the paper \cite{Inahama}.

Let $F$ be a smooth random variable in the sense of Malliavin calculus. Recall that there is an isometry $\mathcal{K}: \bch{([0,1])}\to\ch([0,1])$. For the simplicity of notation,  in the rest of this subsection we will write $\D_h F$ for $\D_{\mathcal{K}(h)}F$ when $h\in\bch([0,1])$, and similarly for higher order directional Malliavin derivatives. 

For any $h_i\in\bch([0,1]), i=1,...,k$, the $k$-th Malliavin derivative of $F$ along the directions of $h_i$, denoted by
$$\D_{h_1,...,h_k}F,$$
is a $k$-linear form on the space $\bch([0,1])$.  Now introduce an $d$-dimensional fractional Brownian motion $\hat{B}$ independent of the original $B$. Suppose there is a natural way to extend $\D_{h_1,...,h_k}F$ from $\bch([0,1])$ to the sample paths of $\hat{B}$; that is,  there is a natural way to define
$$\D_{\hat{B},...,\hat{B}}F$$
for almost all sample paths of $\hat{B}$.  From now on, we fix a sample path $\omega$ of the original fractional Brownian motion $B$. Then, 
$$\hat{\D}_{h_1,...,h_k}(\D_{\hat{B},...,\hat{B}}F(\omega))=k\D_{h_1,...,h_k}F(\omega),$$
where $\hat{\D}$ is the Malliavin derivative along $h_1,..,h_k$ with respect to $\hat{B}$.  Further suppose that $\D_{\hat{B},...,\hat{B}}F(\omega)$ is in the $k$-th Wiener chaos of $\hat{B}$. By equivalence of norms, one then has
\begin{align*}
\|\D^kF(\omega)\|_{\ch([0,1])^{\otimes k}}&=\frac{1}{k}\big(\hat{\me}\|\hat{\D}(\D_{\hat{B},...,\hat{B}}F(\omega))\|^2_{\mathcal{H}([0,1])^{\otimes k}}\big)^{\frac{1}{2}}\\
&\leq \frac{1}{k}\|(\D_{\hat{B},...,\hat{B}})F(\omega)\|_{\hat\D^{2,2}}\leq C_k\|\D_{\hat{B},...,\hat{B}}F(\omega)\|_{L^2(\hat{\mp})}.
\end{align*}
In the above, $\|\cdot\|_{\hat\D^{2,2}}$ is the $(2,2)$-norm with respect to the fractional Brownian motion $\hat{B}$. Therefore in order to bound the $(k,p)$-norm of $F$, one only need to bound the $L^p$-norm (under $\hat{\mp}$) of  $\|\D_{\hat{B},...,\hat{B}}F(\omega)\|_{L^2(\hat{\mp})}$.

\begin{lemma}\label{lem: Malliavin derivative}Let $U_t$ be the log-signature process. We have
	\begin{equation}\label{M derivative of u}
		\|\D U_t\|_{k,p}\leq C_{k,p} t^{H}.
	\end{equation}
\end{lemma}
\begin{proof}
Let $\hat{B}$ in a $d$-dimensional fractional Brownian motion independent of $B$. We  take $F$ in the above argument to be the signature $X_t$ of $B$. At the $m$-th tensor level, the component of $X_t$ is
$$I_t^m=\int_{0\leq t_1\leq\cdots\leq t_m\leq t}dB_{t_1}\otimes\cdots\otimes dB_{t_m}.$$
Hence,
$$\D^2_{h_1,h_2} I^m_t=\sum_{i,j=1}^m\int_{0\leq t_1\leq\cdots\leq t_m\leq t}dB_{t_1}\otimes\cdots\otimes dh_{1,t_i}\otimes\cdots\otimes dh_{2,t_j}\otimes\cdots\otimes dB_{t_m}.$$
Clearly, the above can be extended from $h$ to the sample paths of $\hat{B}$, and we obtain
$$\Xi(t,\hat{B},B)\triangleq\D_{\hat{B},\hat{B}} I^m_t=\sum_{i,j=1}^m\int_{0\leq t_1\leq\cdots\leq t_m\leq t}dB_{t_1}\otimes\cdots\otimes d\hat{B}_{t_i}\otimes\cdots\otimes d\hat{B}_{t_j}\otimes\cdots\otimes dB_{t_m}.$$
By the rescaling property of fractional Brownian motions, we have
$$\Xi(t,\hat{B},B)\stackrel{law}{=}t^{mH}\Xi(1,\hat{B},B).$$
Therefore, by the argument of Inahama, we have
$$\|\D I^m_t\|_{2,p}\leq C_{p} t^{mH}.$$
The higher order Malliavin derivative of $I_t^m$ can be treated similarly, and we have 
$$\|\D I^m_t\|_{k,p}\leq C_{k,p} t^{mH}.$$
Therefore, 
$$\|\D X_t\|_{k,p}\leq C_{k,p} t^{H}.$$
Since
\[
U_t=[\exp^{-1}( X_t)]_\mathcal{B} = \sum_{I, d(I) \le N } \Lambda_I (B)_t [e_{I}]_\mathcal{B},
\] 
we obtain a similar estimate for $U_t$,
\begin{align*}\|\D U_t\|_{k,p}\leq C_{k,p} t^{H}.\end{align*}
The proof is thus completed.
\end{proof}

\medskip
\noindent{\bf Estimate of Malliavin matrix.}\ \  
The bulk of the work in estimating the Malliavin matrix of $U_t$ is done in Theorem \ref{integrability small e-v}.  The lemma below is a direct corollary.
\begin{lemma}\label{lem: Malliavin matrix}
Let $\gamma_t(U)$ be the Malliavin matrix of $U_t$. There exist  constants $C>0$ and $\alpha>0$, all depending on $n$, such that 
\begin{equation} 
	\|\gamma_t(U)^{-1}\|_{n,2^{n+2}}\leq \frac{C}{t^{\alpha}},
\end{equation}
for all $t>0$.
\end{lemma}
\begin{proof}
By Theorem \ref{integrability small e-v}, we have for all $k\geq0$ and $p>1$,
 \begin{align}\label{estimate M derivative and matrix}\|\det\gamma_t(U)^{-1}\|_p\leq \frac{C}{t^{2HNn}},\quad\mathrm{and}\quad \|\mathbf{D}U_t\|_{k,p}\leq C_{k,p}t^H.\end{align}
As a direct consequence, we have for any $(i,j)$-th entry of $\gamma_t(U)^{-1}$,
\begin{align}\label{M matrix ij entry}
\left\|\left(\gamma_t(U)^{-1}\right)^{i,j}\right\|_p\leq C_{k,p,n}\frac{t^{2H(n-1)}}{t^{2HNn}}=\frac{C_{k,p,n}}{t^{2H(Nn-n+1)}}.
\end{align}
  The rest of the proof follows by repeatedly employing \eqref{estimate M derivative and matrix}, \eqref{M matrix ij entry} and  the identity
 $$\mathbf{D}(\gamma_t(U)^{-1})^{ij}=-\sum_{k,l=1}^n(\gamma_t(U^{-1})^{ik}(\gamma_t(U)^{-1})^{lj}\mathbf{D}\gamma_t(U)^{kl}.$$ 
 [More details can be found in \cite{BNOT16} (equation (34) and the paragraph after it).]
\end{proof}
\begin{remark}
We could have been more careful and obtain a more explicit representation for $\alpha$. But for our purpose, knowing the existence of such an $\alpha$ is enough.
\end{remark}

\medskip

Applying Lemma \ref{lem: tail estimate}, Lemma \ref{lem: Malliavin derivative} and Lemma \ref{lem: Malliavin matrix} to (\ref{formula density}), we have, for some constant $\beta>0, C>0$,
\begin{align}\label{upper bound 1}p_t(u)\leq \frac{C}{t^\beta} e^{-\frac{\|u\|_{\textsc{CC}}^2}{Ct^{2H}}}, \quad\mathrm{for\ all}\ u\in\mathbb R^n.\end{align}
With this preliminary bound  for the density $p_t(u)$ of $U_t$, together with the self-similarity of fractional Brownian motions, we are able to sharpen the exponent $\beta$ to the following form.

\begin{theorem}
Denote by $p_t(u)$ the density of $U_t$ with respect to the Lebesgue measure of $\mathbb{R}^n$. There exists a constant $C>0$ such that 
\begin{align*}
p_t(u) \leq  \frac{C}{t^{\nu/2}} e^{-\frac{\|u\|_{\textsc{CC}}^2}{Ct^{2H}}}, \quad\mathrm{for\ all}\ u\in \mathbb{R}^n.
\end{align*}
Here
$\nu=\sum_{i=1}^N i\dim(\mathcal{V})_i$ is the Hausdorff dimension of $\frak{g}_N(\mr^d)$. 
\end{theorem}
\begin{proof}
By the self-similarity of the fractional Brownian motion
\begin{align*}
\triangle_{t^{H}} U_1\stackrel{law}{=} U_t.
\end{align*}
We thus have
\begin{align}\label{density cov}p_t(u)=\frac{1}{t^{\nu/2}}p_1\left(\triangle_{t^{-H}}u\right).\end{align}
The desired result then follows from applying (\ref{upper bound 1}) to the right hand-side of (\ref{density cov}) with $t=1$ and the fact that $\|\cdot\|_{\textsc{CC}}$ is homogeneous with respect to the dilation.
\end{proof}

\section{Positivity of the density} This section is devoted to proving that $p_t(u)$ is strictly positive for all $u\in\mr^n$. We start with some preparations related to the Cameron-Martin space of $B$.

\subsection{Cameron-Martin space of fBm} Denote by $\bch([0,T])$ the Cameron-Martin space associated to a fractional Brownian motion over $[0,T]$. For a smooth function $f:[0,T]\to\mr$, define
$$C^2(f)=\|f\|_{\infty;[0,T]}+\|f'\|_{\infty;[0,T]}+\|f''\|_{\infty;[0,T]}.$$

\begin{lemma}\label{lem: continuous embedding when H<1/2}
For $H<1/2$, the inclusions $\mathcal{H}\subseteq L^2([0,1])$ and $W_0^{1,2}\subseteq \bar{\mathcal{H}}$ are continuous embeddings.
\end{lemma}
\begin{proof}
This is the content of \cite[Lemma 2.3]{OTG19}.
\end{proof}

\begin{lemma}\label{lem: quasi-inverse of signature map}For any $M>0,$ there exists a constant $C_{N,M}>0$\,, such that for every $u\in G_N(\mr^d)$ with $\|u\|_{\textsc{CC}}\leqslant M$, we can find a smooth path $\gamma:[0,1]\rightarrow\mathbb{R}^d$ which satisfies: 
\\
(i) $S_N(\gamma)_1=u$;
\\
(ii) $\dot{\gamma}$ is supported on $[1/3,2/3]$;
\\
(iii) $\|\ddot{\gamma}\|_{\infty;[0,1]}\leqslant C_{N,M}$; and hence
\\
(iv) $\|\dot{\gamma}\|_{\infty;[0,1]}\leqslant C_{N,M}$.
\end{lemma}
\begin{proof}
This is a restatement of \cite[Lemma 4.4]{OTG19}.
\end{proof}

\begin{proposition}\label{Prop: concatenation}
Fix $H\in(0,1)$ and $m\geq2$.  For each $1\leq k\leq m$, let $h_k$ be a smooth path over the interval $[0,T_k]$. Set $T=T_1+...+T_m$. We have, 
\begin{itemize}
\item[(i)]\quad The concatenation of $h_1,...,k_m$, denoted by $h$, is an element in $\bch([0,T])$;
\item[(ii)]\quad The Cameron-Martin norm $\|h\|_{\bch([0,T])}$ is bounded by a constant only depending on  $H$, and $T_k, C^2(h_k)$, for $k=1,...,m$. 
\end{itemize}
\end{proposition}

\begin{proof}
In order to avoid any possible confusion in notation, in the proof we will write $h(s)$ instead of $h_s$ for a path. We also only prove the proposition for $m=2$. The general case follows from a similar argument. 

Let $h_1$ (respectively, $h_2$) be a smooth path over the interval $[0,T_1]$ (respectively, $[0,T_2]$).  First note that $h=h_1\sqcup h_2$ is in $W^{1,2}_0([0,T])$. When $H\leq 1/2$, by Lemma \ref{lem: continuous embedding when H<1/2}, it is clear that $h\in\bch_{[0,T]}$ and $\|h\|_{\bch([0,T])}$ is bounded by a constant only depending on $T_k, C^2(h_k)$, for $k=1,2$. 

In the following, we prove our  result for $H>1/2$.  Recall that $K_H$ is the isometry from $L^2[0,T]$ to $\bch([0,T])$ introduced in Section \ref{intro Malliavin calculus}.  According to Theorem 3.1 of \cite{DU97}, when $H>1/2$ one can expressed $K_H$ in terms of fractional calculus by
$$K_H\phi=C_H\cdot I_{0+}^1\left(t^{H-\frac{1}{2}}\cdot I_{0+}^{H-\frac{1}{2}}\left(s^{\frac{1}{2}-H}\phi(s)\right)(t)\right),\quad \phi(s)\in L^2([0,T]).$$
In order to show $h=h_1\sqcup h_2\in \bch([0,T])$ and bound its Cameron-Martin norm, by inverting the above identity, we only need to show 
\begin{align}\varphi(t)&=t^{H-\frac{1}{2}}D_{0+}^{H-\frac{1}{2}}\left(s^{\frac{1}{2}-H}h'(s)\right)(t)\nonumber\\
&=t^{H-\frac{1}{2}}\left(t^{1-2H}{{h'}}(t)+\left(H-\frac{1}{2}\right)\int_{0}^{t}\frac{t^{\frac{1}{2}-H}h'(t)-s^{\frac{1}{2}-H}h'(s)}{(t-s)^{H+\frac{1}{2}}}ds\right)\nonumber\\
&=I_1(t)+I_2(t)\in L^2([0,T]),\label{phi}
\end{align}
and bounded the $L^2$-norm of $I_1$ and $I_2$ in terms of $C^2(h_k), T_k; k=1,2.$
In the above
$$I_1(t)=t^{\frac{1}{2}-H}h'(t),$$
and
$$I_2(t)=\left(H-\frac{1}{2}\right)t^{H-\frac{1}{2}}\int_{0}^{t}\frac{t^{\frac{1}{2}-H}h'(t)-s^{\frac{1}{2}-H}h'(s)}{(t-s)^{H+\frac{1}{2}}}ds.$$

The estimate for  $\|I_1\|_{L^2;[0,T]}$ is trivial given that $h'(t)$ is uniformly bounded.  In order to estimate $I_2$, we divide our analysis according to $t\in[0,T_2]$ and $t\in(T_1,T]$. 

\medskip
\noindent\underline{\it Estimate of $I_2$ on $[0,T_1]$:} Note that $h(t)=h_1(t)$ when $t\in[0,T_1]$, hence
\begin{align*}
|I_2(t)|&=C_H \left|\int_0^t\frac{h_1'(t)-h'_1(s)}{(t-s)^{H+\frac{1}{2}}}ds+t^{H-\frac{1}{2}}\int_0^t\frac{(t^{\frac{1}{2}-H}-s^{\frac{1}{2}-H})h'_1(s)}{(t-s)^{H+\frac{1}{2}}}ds\right|\\
&\leq C_HC^2(h_1)\left(\int_0^t (t-s)^{\frac{1}{2}-H}ds+ t^{H-\frac{1}{2}}\left|\int_0^t\frac{(t^{\frac{1}{2}-H}-s^{\frac{1}{2}-H})}{(t-s)^{H+\frac{1}{2}}}ds\right|\right)\\
&=C_HC^2(h_1)\big(Q_1(t)+Q_2(t)\big),\quad\quad  \text{for}\ t\in[0,T_1].
\end{align*}
Elementary computation shows that $Q_1(t)\in L^2([0,T_1])$ and its $L^2([0,T_1])$-norm is bounded by a constant only depending on $C^2(h_1)$ and $T_1$. For $Q_2(t)$, we have, by a change of variable $s=ut$,
\begin{align*}
Q_2(t)= t^{\frac{1}{2}-H}\left|\int_0^1\frac{1-u^{\frac{1}{2}-H}}{(1-u)^{H+\frac{1}{2}}}du\right|=C_Ht^{\frac{1}{2}-H}\in L^2([0,T_1]).
\end{align*}
Apparently, we also have $\|Q_2\|_{L^2; [0,T_1]}$ being bounded by a constant only depending on $H$ and $T_1$.

To summarize, when $t\in[0,T_1]$, we can bound  $\|I_2\|_{L^2; [0,T_1]}$ by a constant only depending on $H, C^2(h_1)$ and $T_1$.

\medskip
\noindent\underline{\it Estimate of $I_2$ on $(T_1,T]$:}  In this case, we write
\begin{align*}
I_2(t)&=C_H\left( t^{H-\frac{1}{2}}\int_0^{T_1}\frac{t^{\frac{1}{2}-H}h'(t)-s^{\frac{1}{2}-H}h'(s)}{(t-s)^{H+\frac{1}{2}}}ds+t^{H-\frac{1}{2}}\int_{T_1}^{t}\frac{t^{\frac{1}{2}-H}h'(t)-s^{\frac{1}{2}-H}h'(s)}{(t-s)^{H+\frac{1}{2}}}ds\right)\\
&=C_H\big(J_1(t)+J_2(t)\big),\quad\quad \text{for}\ t\in(T_1,T].
\end{align*}

The estimate for $J_2$ is easy. Since both $s$ and $t$ are in $(T_1,T]$\,,
\begin{align*}
|J_2(t)|\leq C^2(h_2) t^{H-\frac{1}{2}}\int_{T_1}^t\frac{(t-s)}{(t-s)^{H+\frac{1}{2}}}ds=C_HC^2(h_2)t^{H-\frac{1}{2}}(t-T_1)^{\frac{3}{2}-H}.
\end{align*}
Therefore it is in  $L^2([T_1,T])$ with its $L^2$-norm bounded by a constant only depending on $H$, $C^2(h_2)$ and $T_1$.

The analysis for $J_1$ requires some more work, though not difficult.  Note that $|h'(s)|\leq C^2(h_1)$ for $s\in[0, T_1]$ and $|h'(t)|\leq C^2(h_2)$ for $t\in[T_1,T]$.  Thus
\begin{align}
|J_1(t)|&\leq t^{H-\frac{1}{2}}\left(C^2(h_2)C_{T_1}\int_0^{T_1}\frac{1}{(t-s)^{H+\frac{1}{2}}}ds+C^2(h_1)\int_0^{T_1}\frac{s^{\frac{1}{2}-H}}{(t-s)^{H+\frac{1}{2}}}ds\right)\nonumber\\
&\leq t^{H-\frac{1}{2}}\left[C^2(h_2)C_{T_1, H}\big((t-T_1)^{\frac{1}{2}-H}+t^{\frac{1}{2}-H}\big)\right.\nonumber\\
&\left.\quad\quad+C^2(h_1)\left(\int_0^{T_1/2}\frac{s^{\frac{1}{2}-H}}{(t-s)^{H+\frac{1}{2}}}ds+\int_{T_1/2}^{T_1}\frac{s^{\frac{1}{2}-H}}{(t-s)^{H+\frac{1}{2}}}ds\right)\right]\label{J1}.
\end{align}
Because $t>T_1$, we have
\begin{align}\label{second term J1}
\int_0^{T_1/2}\frac{s^{\frac{1}{2}-H}}{(t-s)^{H+\frac{1}{2}}}ds\leq C_{T_1},
\end{align}
for some constant $C_{T_1}>0$. Moreover,
\begin{align}
\int_{T_1/2}^{T_1}\frac{s^{\frac{1}{2}-H}}{(t-s)^{H+\frac{1}{2}}}ds&\leq C_{T_1}\int_{T_1/2}^{T_1}\frac{1}{(t-s)^{H+\frac{1}{2}}}ds\nonumber\\
&\leq C_{T_1}\int_{0}^{T_1}\frac{1}{(t-s)^{H+\frac{1}{2}}}ds\leq C_{T_1, H}\big((t-T_1)^{\frac{1}{2}-H}+t^{\frac{1}{2}-H}\big).\label{third term J1}
\end{align}
Plugging \eqref{second term J1} and \eqref{third term J1} into \eqref{J1}, we obtain, for $t\in(T_1,T]$,
\begin{align*}
|J_1(t)|\leq C_{T_1,H}\,t^{H-\frac{1}{2}}\bigg(C^2(h_1)+C^2(h_2)\bigg)\bigg((t-T_1)^{\frac{1}{2}-H}+t^{\frac{1}{2}-H}+1\bigg).
\end{align*}
It is now clear that $J_1(t)$ is in $L^2([T_1,T])$ with its corresponding $L^2$ bounded above by a constant only depending on $H$ and $C^2(h_k), T_k; k=1,2.$

Based on our analysis on both $J_1$ and $J_2$ above, we conclude that $I_2$ is in $L^2([T_1,T])$ with its corresponding $L^2$-norm bounded by a constant depending only on $H$ and $C^2(h_k), T_k; k=1,2.$

Now that we have finishes our analysis on $I_1$ and $I_2$ in \eqref{phi}, the proof is completed.

\end{proof}

\begin{lemma}\label{lem: CM scaling}
Let $0<T_1<T_2$, and $H\in(0,1)$. Given  $h\in\bar{\mathcal{H}}([0,T_1])$, define $$\tilde{h}_t =h_{T_1 t/T_2}, \quad\quad0\leqslant t\leqslant T_2.$$ Then $\tilde{h}\in\bar{\mathcal{H}}[0,T_2]$, and 
\begin{align*}
\|\tilde{h}\|_{\bar{\mathcal{H}}([0,T_{2}])}=\left(\frac{T_{1}}{T_{2}}\right)^{H}\|h\|_{\bar{{\mathcal H}}([0,T_{1}])}.
\end{align*}
\end{lemma}
\begin{proof}
The lemma is proved in \cite[Lemma 4.3]{OTG19} for $H>1/2$.  Here we provide an intrinsic proof that works for all $H\in(0,1)$.

Fix any $u\in[0,T_1]$ and set
\begin{align}\label{represent h}h_t=\me B_tB_u,\quad \quad t\in[0,T_1].\end{align}
It is clear that $h\in\bch([0,T_1])$ and $\|h\|^2_{\bch([0,T_1])}=\me B_u^2$. In this case, by the self-similarity of fBm, we have
\begin{align*}
\tilde{h}_t=\me B_{T_1t/T_2}B_u=\left(\frac{T_1}{T_2}\right)^{2H}\me B_tB_{uT_2/T_1},\quad\quad t\in[0,T_2].
\end{align*}
Therefore, 
$$\|\tilde{h}\|^2_{\bch([0,T_2])}=\me\left[\left(\frac{T_1}{T_2}\right)^{2H}B_{uT_2/T_1}\right]^2=\left(\frac{T_1}{T_2}\right)^{2H}\|h\|^2_{\bch([0,T_1])}.$$
The proof of the lemma then follows from the fact that $h$ of the form in \eqref{represent h} form a dense subset of $\bch([0,T_1])$.
\end{proof}

\subsection{Positivity of density}
We first introduce some notations in order to state a general criterion for the positivity of density of a non-degenerate random vector $F=(F^1,...,F^n)$. 

For a given element $\underline{\ell}=(\ell_1,\ldots,\ell_n)\in\bch^n([0,1])$ and a vector $z\in\mr^n$, the shifted fractional Brownian motion is given by
$$T_z^{\underline{\ell}}B=B+\sum_{j=1}^nz_jl_j.$$
For any multi-index $\alpha=(\alpha_1,\ldots,\alpha_k)$ lying in $\{1,2,\ldots,n\}^k$, let $\underline{\ell}_\alpha=(\ell_{\alpha _1},\ldots,\ell_{\alpha _k})$ and define
$$R_{\underline{\ell}_\alpha,p}F=\int_{\{|z|\leq1\}}    {_{\ch^{\otimes k}}}\left\langle(\mathbf{D}^kF)(T_z^{\underline{\ell}}B), \ell_{\alpha _1}\otimes\cdot\cdot\cdot\otimes \ell_{\alpha _k}\right\rangle^p_{\bch ^{\otimes k}}dz,$$
for some $p>n$ and multi-index $\alpha$ with $|\alpha|=k\geq 0$.

The following criterion is borrowed from \cite[Theorem 3.1]{BNOT16}, and summarizes the content of Section 4.2 of \cite{B-N Saint-Flour}.

\begin{theorem}\label{th: positivity-sufficient}
Let $F=(F^1,\ldots,F^n)$ be a non-degenerate random variable and $\Upsilon:\bch([0,1]) \to\mr^n$ a $\mathcal{C}^{\infty}$ functional. Suppose that the following conditions hold:
\begin{itemize}
\item[a.] For any $h\in\bch([0,1]) $ there exists a sequence of measurable transformations $T^h_N: \Omega\to\Omega$ such that $\mp\circ(T^h_N)^{-1}$ is absolutely continuous with respect to $\mp$; 
\item[b.] Let $\{\mathbf{D}\Upsilon^{j}(h) ; \, j=1,\ldots, n\}$ be the coordinates of $\mathbf{D}\Upsilon(h)$ in $\mr^{n}$, and set 
$$\underline{\ell}=(\mathbf{D}\Upsilon^1(h),\ldots,\mathbf{D}\Upsilon^n(h)).$$ Suppose that for every $\varepsilon>0$:
\begin{enumerate}
\item $\lim_{N\to\infty}\mp\{|F\circ T^h_N-\Upsilon(h)|>\varepsilon\}=0$;
\item $\lim_{N\to\infty}\mp\{\|(\mathbf{D}F)\circ T^h_N- (\mathbf{D}\Upsilon)(h)\|_{\ch }>\varepsilon\}=0$; and
\item $\lim_{M\to\infty}\sup_{N}\mp\{(R_{\underline{\ell}_\alpha,p}F)\circ T^h_N>M\}=0$ for some $p>n$ and all multi-index $\alpha$ with $|\alpha|=0,1,2,3.$
\end{enumerate}
\item[c.] Finally, for a fixed $y\in\mr^n$ assume that there exists  an $h\in\bch([0,1]) $ such that $\Upsilon(h)=y$ and for the deterministic Malliavin matrix $\gamma_{\Upsilon}(h)$ of $\Upsilon$ at $h$, one has $\det \gamma_{\Upsilon}(h)>0$. 
\end{itemize}
Then the density of $F$ at $y$ satisfies $p(y)>0$.
\end{theorem}
Several remarks regarding the above theorem are in order and listed below.
\begin{remark}\label{def of T^h_N}
We denote by $\{\Pi^N, N\geq 1\}$ a sequence projections from $\Omega$ to $\bch([0,1])$ of finite-dimensional range, which converges strongly to the identity. For any $h\in\bch([0,1])$, we simply define $T_N^h$ by 
$$T_N^h(B)=B-\Pi^NB+h.$$
The existence of such a sequence of projections $\Pi^N$ so that $\mp\circ(T^h_N)^{-1}$ is absolutely continuous with respect to $\mp$ is proved in Corollary 2.8 of \cite{BNOT16}. Hence, item (a) in Theorem \ref{th: positivity-sufficient} is satisfied. 
\end{remark}

\begin{remark}\label{regarding item b}
Recall that $U_t$ is the log-signature process that satisfies the SDE \eqref{SDE u}.  We aim to apply Theorem \ref{th: positivity-sufficient} to $F=U_1$.   A natural choice of $\Upsilon$ is $\Upsilon(\cdot)=\Psi(\cdot)_1$, where  $\Psi$ is the It\^{o}-Lyon maps associated to \eqref{SDE u}. With such choice of $F$, $\Upsilon$ and $T^h_N$ given in Remark \ref{def of T^h_N}, it is proved in Theorem 1.4 of \cite{BNOT16} that item (b) of Theorem \ref{th: positivity-sufficient} is satisfied for equation \eqref{SDE u} if the vector fields $V_i$'s are $C^\infty$-bounded. Although in our current situation, $V_i$'s are not bounded (of polynomial order), it is not hard to check that the same argument also work. Indeed, it is an simple consequence of the following facts:
\begin{itemize}
\item[(i)]  Denote by $\bar{\mathbf{B}}^N$ the signature of\, $T^h_N(B)=B-\Pi^N B+h$, and $\mathbf{h}$ the signature of $h$. it can be shown that for all $q>1$, we have $d_{p-var;[0,1]}(\bar{\mathbf{B}}^N, \mathbf{h})\to 0$ in $L^q(\mp)$, as $N\to\infty.$
\item[(ii)] Components of $U_1\circ T^h_N(B)=\log S_N(T^h_N(N))_1$ are simply polynomials of components of the signature of $T_N^h(B)$.
\end{itemize}
\end{remark}



 \begin{remark}\label{regarding item c}
Item (c) in Theorem \ref{th: positivity-sufficient} says that one should have an $h\in \bch([0,1])$ such that $\Upsilon(h)=y$ and  $\Upsilon$ is a submersion at $h$.
\end{remark}

According to Theorem \ref{th: positivity-sufficient}, and thanks to Remark \ref{regarding item b} and Remark \ref{regarding item c}, in order to show the positivity of the density $p_1(u)$ of $U_1$, it suffices to show that for any $u\in\mr^n$, there exists an $h\in\bch([0,1])$ such that $\Psi(h)_1=u$ and $\Psi(\cdot)_1$ is a submersion at $h$.  This will be the main content of the rest of this section.

We start our discussion on $\hG_N(\mr^d)$. Let $\bch_0([0,1])$ be the space of piecewise linear functions from $[0,1]$ to $\mr^d$. For any $h\in\bch_0([0,1])$, denote by$S_N(h)_t$ the signature process of $h$ on the free Nilpotent group of level $N$ over $\mr^d$. That is,
$$S_N(h)_t=\left(1,\int_0^t dh_s, \int_{0\leq s_1<s_2\leq t} dh_{s_1}\otimes dh_{s_2},...,\int_{0\leq s_1<\cdots<s_N\leq t}dh_{s_1}\otimes\cdots\otimes dh_{s_N}\right)\in \hG_N(\mr^d). $$
Chow's theorem states that 
$$\left\{S_N(h)_1, h\in\bch_0([0,1])\right\}=\hG_N(\mr^d).$$
Moreover, by Proposition \ref {Prop: concatenation} and Lemma \ref{lem: CM scaling}, we know that $\bch_0([0,1])\subset\bch([0,1])$ for all $H\in(0,1)$.

\begin{theorem}\label{immersion} For any $g\in \hG_N(\mr^d)$, we can find a path $h\in\bch([0,1])$ such that $S_N(h)_1=g$ and the map $S_N(\cdot)_1: \bch\to \hG_N(\mr^d)$ is non-degenerate at $h$.
\end{theorem}
\begin{proof}
As in the proof of Proposition \ref{Prop: concatenation}, in order to avoid any possible confusion in notation, we will use $h(s)$ instead of $h_s$ for a path throughout this proof.

Let $n=\mathrm{dim}\,\hG_N(\mr^d)$. It suffices to show that for any $g\in \hG_N(\mr^d)$ there exists an $h\in\bch([0,1])$ and $n$ families of paths $\{h^{k,\epsilon}_\cdot\in\bch([0,1]), \epsilon\in[0,1]\}, k=1,...,n,$ such that 
\begin{itemize}
\item [(1)] $S_N(h)_1=g$;
\item [(2)] $h^{k,0}=h$ for all $k=1,...,n$;
\item [(3)] $\frac{d h^{k,\epsilon}}{d\epsilon}\big|_{\epsilon=0}\in\bch([0,1])$ for all $k=1,...,n$;  and
\item [(4)] $\left\{\frac{dS_N(h^{k,\epsilon})_1}{d\epsilon}\big|_{\epsilon=0}, \ k=1,...,n\right\}$ spans the tangent space of $\hG_N(\mr^d)$ at $g$; that is, they are linearly independent. 
\end{itemize}

We divide the proof into two steps.

\noindent{\it Step 1:} We show that the above is true for at least one $g_0\in G_N(\mr^d)$. Let $\{e_1,...,e_d\}$ be the standard basis of $\mr^d$, and recall $n=\mathrm{dim}\, G_N(\mr^d)$.  By Lemma 3.32 of \cite{ABB2012}, there exists $e_{i_1},..., e_{i_n}$ and $\bar{s}_1,..., \bar{s}_n\in\mr$ such that the map
$$\phi: \mr^n\to G_N(\mr^d),\quad \phi(s_1,...,s_n)=e^{s_1e_{i_1}}\otimes\cdots\otimes e^{s_ne_{i_n}},$$
is non-degenerate at $\bar{s}=(\bar{s}_1,...,\bar{s}_n)$. 

Set $g_0=\phi(\bar{s}_1,...,\bar{s}_n)$, and let $\gamma_k$ be the straight line in $\mr^d$ over the time interval $[0,1/n]$ whose tangent is $n\bar{s}_ke_{i_k}; k=1,...,n$. Define $$h_0=\gamma_1\sqcup \gamma_2\sqcup\cdots\sqcup \gamma_n\in\bch_0([0,1]),$$ the concatenation of $\gamma_k; k=1,...,n$. Moreover, for each $1\leq k\leq n$, define $\gamma_k^\epsilon=(1+\epsilon)\gamma_k$ and construct $h_0^{k,\epsilon}$ the same way as $h_0$ but replacing $\gamma_k$ by $\gamma^\epsilon_k$ in the concatenation. It is then easy to check all the four properties are satisfied at this particular choice of $g_0\in \hG_N(\mr^d)$. Indeed, (1) and (2) follows easily by the very construction of $h_0$ and $h_0^{k,\epsilon}$. (4) is satisfied by the fact that
$$\left.\frac{dS_N(h_0^{k,\epsilon})_1}{d\epsilon}\right|_{\epsilon=0}=\frac{\partial\phi}{\partial s_k}\bigg|_{\bar{s}},$$
and that $\phi$ is non-degenerate at $\bar{s}$. For (3), note that for each $k=1,...,n,$
\begin{align*}
\left.\frac{d h_0^{k,\epsilon}}{d\epsilon}\right|_{\epsilon=0}(t)=\left\{
\begin{array}{lll}
0,\quad&\mathrm{if}\ t\in[0,(k-1)/n];\\
n(t-(k-1)/n)\bar{s}_ke_{i_k},&\mathrm{if}\ t\in[(k-1)/n, k/n];\\
\bar{s}_ke_{i_k},&\mathrm{if}\ t\in[k/n, 1].
\end{array}
\right.
\end{align*}
Clearly it is an element in $\bch_0([0,1])$ and hence in $\bch([0,1])$.  This finishes the proof of Step 1.

\noindent{\it Step 2:} With the help of Step 1, the proof for general $g\in \hG_N(\mr^d)$ relies on the group structure of $\hG_N(\mr^d)$. Indeed, for any $g\in \hG_N(\mr^d)$ we have 
$$g=g_1\otimes g_0,$$ for some $g_1\in \hG_N(\mr^d)$. By Lemma \ref{lem: quasi-inverse of signature map},  one can find a smooth path $h_1\in\bch([0,1])$ such that $g_1=S_N(h_1)_1$. Now let $\tilde{h}=h_1\sqcup h_0$ and $\tilde{h}^{k,\epsilon}=h_1\sqcup h_0^{k,\epsilon}$, and re-parametrize them back to interval $[0,1]$ by
$$h(s)=\bar{h}({2s}),\quad\text{and} \quad h^{k,\epsilon}(s)=\bar{h}^{k,\epsilon}({2s}), \quad s\in [0,1].$$
By Proposition \ref{Prop: concatenation},  $h$ and $h^{k,\epsilon}$ are paths in $\bch([0,1])$. Moreover, it is easy to check that they satisfy all the properties (1)-(4) in the above.
The proof is thus completed. 
\end{proof}

\begin{remark}\label{connecting by nondegenerate path}
Suppose $g\in \hG_N(\mr^d)$ with $\|g\|_{\textsc{CC}}\leq M$ for some constant $M$. Let ${h}$ be the path constructed in the proof of Theorem \ref{immersion} such that $S_N({h})_1=g$ and that $S_N(\cdot)_1: {\bch}([0,1])\to \hG_N(\mr^d)$ is non-degenerate at ${h}$. One can show  $$\|{h}\|_{{\bch}([0,1])}\leq C_M,$$ for some constant $C_M$ only depending on $M$. 

To this aim, first note that we only need to prove $\|\bar{h}\|_{\bch([0,2])}\leq C_M$ thanks to Lemma \ref{lem: CM scaling}.  By the construction of $\bar{h}$ in the proof of Theorem \ref{immersion}, we have
\begin{align}\label{bar h expression}\bar{h}=h_1\sqcup h_0,\end{align}
where $h_1$ is such that $S_N(h_1)_1=g_1=g\otimes g_0^{-1}$. Since $\|g\|_{\text{CC}}\leq M$ and $g_0$ is a fixed element, we have
$$\|g_1\|_{\textsc{CC}}\leq \|g\otimes g_0^{-1}\|_{\textsc{CC}}=\|g\|_{\textsc{CC}}+\|g_0\|_{\textsc{CC}}\leq M+\|g_0\|_{\textsc{CC}}\leq C_{1,M}.$$
Hence, by Lemma \ref{lem: quasi-inverse of signature map},
$$C^2(h_1)\leq C_{2,M},$$
for some constant only depending on $C_{1,M}$ (and hence only on $M$). On the other hand, observe that $h_0$ appearing in \eqref{bar h expression} is a fixed piecewise linear path. [It has $n$ pieces where $n=\dim \hG_N(\mr^d)$.] Hence, $\bar{h}$ is a concatenation of $n+1$ smooth paths whose first and second derivatives for each piece  are bounded by a constant only depending on $M$. Now Proposition \ref{Prop: concatenation} concludes that $\|\bar{h}\|_{{\bch}([0,1])}\leq C_M$.

\end{remark}

The following is an immediate corollary of Theorem \ref{immersion}.
\begin{theorem}\label{immersion Lie algebra} Recall that $\Psi$ is the It\^{o}-Lyons map associated to equation \eqref{SDE u}.  For any $u\in \mr^n$, we can find a path $h\in\bch([0,1])$ such that $\Psi(h)_1=u$ and the map $\Psi(\cdot)_1: \bch\to \mr^n$ is non-degenerate at $h$.
\end{theorem}
\begin{proof}
Denote by $\varphi$ the global isomorphism from $\hG_N(\mr^d)$ to $\mr^n$.  Then $\Psi(\cdot)_1=\varphi (S_N(\cdot)_1 )$. The rest of the proof follows from Theorem \ref{immersion}.
\end{proof}

Thanks to Theorem \ref{th: positivity-sufficient}, Remark \ref{def of T^h_N}, Remark \ref{regarding item b}, Remark \ref{regarding item c}, and Theorem \ref{immersion Lie algebra}, we have proved the main result of this section.
\begin{theorem}\label{th: positivity of density}
Fix $H>1/4$. Let $U_t$ be the log-signature process of the fractional Brownian motion. Denote by $p_t(u)$ the density of $U_t$. We have
$$p_1(u)>0,\quad\text{for\ all}\ u\in \mathbb{R}^n.$$
\end{theorem}

The strict positivity of $p_1(u)$ gives a sharp local lower bound of the density $p_t(u)$ of $U_t$.
\begin{corollary}
With the same notations in Theorem \ref{th: positivity of density}, we have, for some constant $c>0$,
\begin{align}\label{local lower bound}
p_t(u)\geq \frac{c}{t^{\nu/2}},\quad \text{for\ all}\ u\ \text{with}\ \|u\|_{\textsc{CC}}\leq t^H,
\end{align}
where 
$\nu=\sum_{i=1}^N i\dim(\mathcal{V})_i$ is the Hausdorff dimension of $\frak{g}_N(\mr^d)$.
\end{corollary}
\begin{proof}
By \eqref{density cov}, we have for all $u\in \mathbb{R}^n$ with $\|u\|_{\textsc{CC}}\leq t^H$,
\begin{align*}
p_t(u)&=\frac{1}{t^{\nu/2}}p_1\left(\Delta_{t^{-H}}u\right)\\
&\geq \frac{1}{t^{\nu/2}}\inf\{p_1(u): u\in\mathbb{R}^n \  \text{and}\ \|u\|_{\text{CC}}\leq 1\}. 
\end{align*}
The fact that $c=\inf\{p_1(u): u\in \mathbb{R}^n \  \text{and}\ \|u\|_{\text{CC}}\leq 1\}>0$ follows from the continuity and strict positivity of $p_1(u)$.
\end{proof}

\bigskip
\section{Varadhan estimate} 
In this section, we establish the Varadhan estimate for the signature of fractional Brownian motions. Moreover, we will show that the controlling ``distances" that appear in the Varadhan estimate are both equivalent to the C-C distance. This is consistent with the results obtained in the previous sections for the upper and lower bounds of the density function.

\subsection{Varadhan estimate} 
We consider the log-signature process $U_t\in \mathbb{R}^n$. Recall that it satisfies a canonical SDE
\begin{align}\label{SDE varadhan estimate}U_t=\sum_{i=1}^d\int_0^tV_i(U_s)dB^i_s,\end{align}
where $V_i$ are smooth vector fields with polynomial growth. Denote by $\Psi$ the It\^{o}-Lyons map associated to the above equation.  Then we can write $U_t=\Psi(B)_t$. Define 
\begin{align}\label{def: d1}d(u)&=\inf\{ \|h\|_{\bch([0,1])}: h\in\bch([0,1]), \text{and}\ \Psi(h)_1=u\}.
\end{align}
This is the controlling `distance' between $u$ and $0$. Another `distance' of interest  in the Varadhan type estimate and of similar spirit is given as follows
\begin{align}\label{def: dr1}d_R(u)=\inf\{ \|h\|_{\bar\cH}: h\in\bar{\cH},  \Psi(h)_1=u, \text{and}\  \langle \D\Psi(h),\D\Psi(h)\rangle_{\cH} \text{is\ non-degenerate}.\}.\end{align}
{When an SDE of the form in \eqref{SDE varadhan estimate} is driven by a standard Brownian motion, it is shown in \cite[Theorem 1.1]{Leandre PTRF87} that , under strong H\"{o}rmander conditions, the above two controlling distances $d$ and $d_R$ are identical. The argument crucially relies on the $L^2$ structure of the Cameron-Martin space of the Brownian motion. It is not clear how it can be adapted in the case of fractional Brownian motions. However, we will show below that both $d$ and $d_R$ are  equivalent to the C-C distance.  }
\begin{remark}
When the vector fields in \eqref{SDE varadhan estimate} are uniformly elliptic, one can show that the two controlling distances $d$ and $d_R$ are always the same for a large class of Gaussian processes. We refer the interested reader to \cite[Lemma 4.7]{GOT18} for more details in this regard.
\end{remark}

Consider the following family of stochastic differential equations driven by $B$:
\begin{align}\label{u eps}U_t^\eps=\eps\sum_{i=1}^d\int_0^t V_i(U^\eps_s)dB_s^i, \quad \eps\in(0,1].\end{align}
Our main result of this section is the following.
\begin{theorem}\label{thm: varadhan estimate}
Let $p_\eps(u)$ be the density of $U_1^\eps$. Then
\begin{align}\label{varadhan lower bound}
\liminf_{\eps\downarrow 0} \eps^2\log p_\eps(u)\geq -\frac{1}{2}d^2_R(u),
\end{align}and
\begin{align}\label{varadhan upper bound}
\limsup_{\eps\downarrow0}\eps^2\log p_\eps(u)\leq -\frac{1}{2}d^2(u).
\end{align}
\end{theorem}

We first lay down several lemmas that are crucial in the proof of Theorem \ref{thm: varadhan estimate}.

\begin{lemma}\label{lem: large deviation} $U^\eps_1$ satisfies a large deviation principle with rate function $d(\cdot)/2$, where $d(\cdot)$ is defined in \eqref{def: d1}.
\end{lemma}
\begin{proof}
Fix any $p>\max\{1/H, N\}$. It is know (see, e.g., \cite{MS}) that $\triangle_\eps\mathbf{B}\triangleq\triangle_\eps S_N(B)$, as a $\hG_{\lfloor p\rfloor}(\mr^d)$-valued rough path satisfies a large deviation principle in $p$-variation topology with good rate function given by
\begin{align*}
J(h)=\left\{\begin{array}{ll}\frac{1}{2}\|h\|^{2}_{\bch([0,1])}\ \mathrm{if}\ h\in\bch([0,1])\\ +\infty\quad \mathrm{otherwise}.
\end{array}
\right.
\end{align*}
Note that $p>N$ and let  $\pi^p_N: \hG_{\lfloor p \rfloor}(\mr^d))\to \hG_{N}(\mr^d))$ be the canonical projection. It is then clear that $\Psi(\cdot)_1=\exp^{-1}\circ \pi^p_N ((\cdot)_1): C^{p-{\rm{var}}}([0,1],\hG_{\lfloor p \rfloor}(\mr^d)) \to \frak{g}_N(\mr^d)$ is continuous.  Here we have identified $\frak{g}_N(\mr^d)$ with $\mr^n$ through the basis $\mathcal{B}$. Now note that $U_1^\eps=\Psi(\triangle_\eps{\bf{B}})_1$. The large deviation principle of $U^\eps_1$ follows from the contraction principle. 
\end{proof}

\begin{lemma}\label{lem: convergence in D infty}For each $h\in\bch([0,1])$, we have
\begin{align}\label{converge of derivative in D infty}
\lim_{\eps\downarrow 0}\frac{1}{\eps}\left(\Psi(\eps B+h)_t-\Psi(h)_t\right)=Z(h)_t,
\end{align}
in the topology of $\mathbb{D}^\infty$. Moreover, $Z(h)_t$ is a centered Gaussian random variable in $\mathbb{R}^n$ with variance $\gamma_{\Psi(h)_t}=\langle \mathbf{D}\Psi(h)_t, \mathbf{D}\Psi(h)_t\rangle_{\bch{([0,1])}}$, the deterministic Malliavin matrix of $\Psi(\cdot)_t$ at $h$.
\end{lemma}
\begin{proof}
It is clear that  $\Psi(\eps{B}+h)_t$ satisfies the following rough SDE
\begin{align}\label{equation for psi ep}
{\Psi(\eps{B}+h)_t=\sum_{i=1}^d\int_0^tV_i(\Psi(\eps{B}+h)_s)d(\eps{B^i}+h_s^i).}\end{align}
By standard path-wise estimates, $\Phi(\eps{\bf{X}}+h)_t$ is smooth in $\eps$ and its derivatives satisfy a rough SDE obtain by formally differentiating \eqref{equation for psi ep} on both sides (see, e.g., \cite[Proposition 11.4]{FV2010}). In particular, at $\eps=0$, we have
\begin{align*}
\lim_{\eps\downarrow0}\frac{1}{\eps}\left(\Psi(\eps {B}+h)_t-\Psi(h)_t\right)=Z(h)_t,
\end{align*}
where $Z(h)_t$ satisfies the rough differential equation 
\begin{align}\label{equation for Zh}
Z(h)_t=\sum_{i=1}^d\int_0^tDV_i(\Psi(h)_s)Z(h)_sdh^i_s 
+\sum_{i=1}^d\int_0^tV_i(\Psi(h)_s)dB^i_s.
\end{align}
Thus, for any $k\in\bch$ the Malliavin derivative of $Z(h)_t$ along the direction $k$ satisfies the rough equation
\begin{align}\label{equation for DZ}
D_kZ(h)_t=\sum_{i=1}^d\int_0^tDV_i(\Psi(h)_s)D_kZ(h)_sdh^i_s 
+\sum_{i=1}^d\int_0^tV_i(\Psi(h)_s)dk^i_s.
\end{align}
Note that equation \eqref{equation for DZ} is deterministic, which implies $Z(h)_t$ is in the first Chaos. The fact that $Z(h)_t$ is also centered can be seen by taking expectation on both sides of equation \eqref{equation for Zh}.   On the other hand, it is easy to check that the deterministic Malliavin derivation of $\Phi(h)_t$ along the direction of $k$ satisfies exactly the same equation \eqref{equation for DZ}. Hence, $Z(h)_t$ is a centered Gaussian random variable with covariance matrix given by $$\gamma_{\Psi(h)_1}=\langle \mathbf{D}\Psi(h)_1, \mathbf{D}\Psi(h)_1\rangle_{\bch{([0,1])}},$$ the deterministic Malliavin matrix of $\Psi(\cdot)_1$ at $h$.

The fact that the convergence in \eqref{converge of derivative in D infty} takes place in $\mathbb{D}^\infty$ can be seen from the following observations.
\begin{itemize}
\item [(a)] Components of the signature $S_N(\eps{B}+h)_t$ are simply iterated integrals of $\eps{B}+h$, and it is easy to see that
\begin{align*}
\frac{1}{\eps}\left(S_N(\eps(B)+h)_t-S_N(h)_t\right)
\end{align*}
converges in $\mathbb{D}^\infty$, using the idea of Inahama explained before to control the $\cH^{\otimes k}$-norm.
\item [(b)] We have $\Psi(\eps{B}+h)=\log (S_N(\eps{B}+h))$, and $\exp^{-1}$ is a smooth map with polynomial growth.
\end{itemize}
The proof is thus completed.
\end{proof}

\begin{lemma}\label{estimate of M derivative and matrix} 
Let $U^\eps_t$ be defined in \eqref{u eps}. We have for $\eps\in(0,1]$,
\begin{align}\label{Malliavin derivative of u}\|\D U_1^\eps\|_{k,p}\leq C_{p} \eps^{H},\end{align}
and
\begin{align}\label{Malliavin matrix of u}
\|\det\gamma_{U_1^\eps}^{-1}\|_p\leq C_p \eps^{-2HnN}.
\end{align}
where $n$ is the dimension of $\frak{g}_N(\mr^d)$ and $C_p$ is a positive constant depending on $p$.
\begin{proof}
This is an easy consequence of \eqref{M matrix of u}, \eqref{M derivative of u} and the self-similarity of $B$.
\end{proof}

\end{lemma}


\bigskip
\begin{proof}[Proof of Theorem \ref{thm: varadhan estimate}] With Lemma \ref{lem: large deviation}, Lemma \ref{lem: convergence in D infty} and Lemma \ref{estimate of M derivative and matrix} in hand,  , the proof of Theorem \ref{thm: varadhan estimate} is standard.  For the sake of completeness, we online the proof below. More details can be found in \cite{BOZ15}.  

\medskip
\noindent\emph{Lower bound:}\quad We first prove \eqref{varadhan lower bound}. To this aim, for any $u\in\mathbb{R}^n$, fix an arbitrary $\eta>0$ and let $h\in\bch$ be such that $\Psi(h)_1=u$, $\gamma_{\Psi(h)_1}=\langle \D\Psi(h),\D\Psi(h)\rangle_{\cH}$ is non-degenerate,  and $\|h\|^2_{\bch}\leq d^2_R(u)+\eta$. Let $f\in C_0^\infty(\mathbb R^n).$ By Cameron-Martin's theorem for $B$, it is readily checked that
$$\me\lc f(U^\eps_1)\rc=e^{-\frac{\|h\|_{\bch}^2}{2\eps^2}} \, 
\me \lc f(\Psi_1(\eps B+h))e^{-\frac{B(h)}{\eps}}\rc,$$ 
where  $B(h)$ denotes the Wiener integral of $h$ with respect to $B$. Now consider a function $\chi\in C^\infty(\mr)$, satisfying  $0\leq \chi\leq 1$, such that $\chi(t)=0$ if $t\not\in[-2\eta, 2\eta]$, and $\chi(t)=1$ if $t\in[-\eta,\eta]$. Then, if $f\geq 0$, we have
$$\me\lc f(U^\eps_1)\rc
\geq e^{-\frac{\|h\|^2_{\ch}+4\eta}{2\eps^2}}\, 
\me\lc \chi(\eps B(h))f(\Psi(\eps B+h)_1)\rc.
$$
Hence, by means of an approximation argument applying the above estimate to $f=\delta_{u}$, we obtain
\begin{align}\label{lower bound claim1}
\eps^2\log p_\eps(u)\geq 
-\left(\frac{1}{2}\|h\|_{\bch}^2+2\eta\right)
+\eps^2\log\me\big[\chi(\eps B(h))\delta_u(\Psi(\eps B+h)_1)\big].
\end{align}

We now bound the right hand side of equation \eqref{lower bound claim1}. Owing to the fact that $\Psi(h)_1=u$ and thanks to the scaling properties of the Dirac distribution, it is easily seen that
\begin{align}\label{varadhan estimate extra term}
\me\big(\chi(\eps B(h))\delta_u(\Psi(\eps B+h)_1)\big)=\eps^{-n}\me\left(\chi(\eps B(h))\delta_0\left(\frac{\Psi(\eps B+h)_1-\Psi(h)_1}{\eps}\right)\right).
\end{align}
Thanks to Lemma \ref{lem: convergence in D infty}, when we send $\eps$ to 0 the expectation on the right hand-side of   \eqref{varadhan estimate extra term} tends to $\me\delta_0(Z(h)_1)$. 
In particular, we get
$$\lim_{\eps\downarrow0}\eps^2\log\me\big(\chi(\eps B(h))\delta_u(\Psi(\eps B+h)_1)\big)=0.$$
Plugging this information in (\ref{lower bound claim1})  and letting $\eps\downarrow0$ we end up with
$$
\liminf_{\eps\downarrow0}\eps^2\log p_{\eps}(u)\geq
-\lp \frac{1}{2}\|h\|^2_{\bch}+2\eta\rp
\geq -\lp d^2_R(u)+3\eta\rp.
$$
Since $\eta>0$ is arbitrary this yields (\ref{varadhan lower bound}).

\bigskip
\noindent\emph{Upper bound:}
Now we show the upper bound \eqref{varadhan upper bound}.
Fix a point $u\in\mathbb{R}^n$ and consider a function $\chi\in C_0^\infty(\mathbb{R}^n), 0\leq\chi\leq1$ such that $\chi$ is equal to one in a neighborhood of $u$. The density of $U_1^\eps$ at point $u$ is given by
$$
p_\eps(u)=\me\lc\chi(U_1^\eps)\delta_u(U_1^\eps)\rc.
$$
Integrate the above expression by parts in the sense of Malliavin calculus (see, e.g., \cite{Nualart06}) and apply H\"{o}lder's inequality (see, e.g., \cite[Proposition 1.5.6]{Nualart06}), we have
\begin{align*}
\me[\chi(U_1^\eps)\delta_u(U_1^\eps)]
\leq&\mp(U_1^\eps\in\mathrm{supp}\chi)^\frac{1}{q}\cdot c_{q}\|\gamma_{U_1^\eps}^{-1}\|_\beta^m\|{\bf{D}} U_1^\eps\|_{n,\gamma}^r\|\chi(U_1^\eps)\|^n_{n,q},
\end{align*}
for some constants $q>1$, $\beta, \gamma>0$ and integers $m, r$. Thus, invoking the estimates  (\ref{Malliavin derivative of u}) and (\ref{Malliavin matrix of u}), we obtain
$$\lim_{\eps\downarrow0}\eps^2\log \me[\chi(U_1^\eps)\delta_u(U_1^\eps)]\leq\lim_{\eps\downarrow0}\eps^2\log \mp(U_1^\eps\in\mathrm{supp}\chi)^\frac{1}{q}.$$
Finally the large deviation principle for $U_1^\eps$  in Lemma \ref{lem: large deviation} ensures that for small $\eps$ we have
$$\mp(U_1^\eps\in\mathrm{supp}\chi)^\frac{1}{q}\leq e^{-\frac{1}{q\eps^2}(\inf_{z\in\mathrm{supp}\chi}d^2(z)+o(1))}.$$
Since $q$ can be chosen arbitrarily close to 1 and $\mathrm{supp}(\chi)$ can be taken arbitrarily close to $y$, the proof of \eqref{varadhan upper bound} is now easily concluded {thanks to the lower semi-continuity of $d$.}

\end{proof}

\subsection{Controlling distance} Comparing the  Varadhan estimate to the results in the upper and lower bound, it is natural to ask whether the controlling distance $d$ and $d_R$ are comparable to the C-C distance. Our answer is affirmative. And this section is devoted to the proof of the equivalence of these three quantities.  Although our argument can be carried out on the Lie algebra $\frak{g}_N(\mr^d)$, we will instead work on the Lie group $\hG_N(\mr^d)$. Because the argument is more direct on $\hG_N(\mr^d)$, and obtained results can be easily translated to corresponding results on $\frak{g}_N(\mr^d)$. 

To this aim, denote by $\Phi(\cdot)_t=S_N(\cdot)_t: \bar{\cH}\to \hG_N(\mr^d)$, the It\^{o} map associated to equation \eqref{SDE: signature}.
 For any $g\in \hG_N(\mr^d)$, define
\begin{align}\label{def: d}\tilde d(g)&=\inf\{ \|h\|_{\bch([0,1])}: h\in\bch([0,1]), \text{and}\ \Phi(h)_1=g\}.
\end{align}
This is the controlling `distance'  (of the system \eqref{SDE: signature}) between $g$ and $\mathbf{1}$, the group identity of $\hG_N(\mr^d)$.  Similarly, set
\begin{align}\label{def: dr}\tilde{d}_R(g)=\inf\{ \|h\|_{\bar\cH}: h\in\bar{\cH},  \Phi(h)_1=g, \text{and}\  \langle \D\Phi(h),\D\Phi(h)\rangle_{\cH} \text{is\ non-degenerate}.\}.\end{align}

\begin{remark}\label{rem: identify distances g and G}
Recall that, with previous notations, for any $h\in\bch([0,1])$, $[ \exp^{-1} (\Phi(h)_t)]_\mathcal{B}=\Psi(h)_t$ and that $\exp: \frak{g}_N(\mr^d)\to \hG_N(\mr^d)$ is a global diffeomorphism. It is clear that for $[ \exp^{-1}(g)]_\mathcal{B}=u$ we have $d(u)=\tilde{d}(g)$ and $d_R(u)=\tilde{d}_R(g)$.
\end{remark}
In the rest of this section, we will show that $\tilde{d}$, $\tilde{d_R}$ and the C-C distance on $\hG_N(\mr^d)$ are all equivalent.
We first state a well-known embedding theorem for the Cameron-Martin space $\bch$ of the fractional Brownian motion.
\begin{lemma}\label{lemma: variational embedding}If $H>\frac{1}{2}$, then
$\bar{\mathcal{H}}\subseteq C_{0}^{H}([0,1];\mathbb{R}^{d}).$ If $H\leqslant\frac{1}{2}$,
then for any $q>\left(H+1/2\right)^{-1}$, we have 
$\bar{\mathcal{H}}\subseteq C_{0}^{q-\mathrm{var}}([0,1];\mathbb{R}^{d}).$ The above inclusions are continuous embeddings.
\end{lemma}

With the help of Lemma \ref{lemma: variational embedding}, we are able to show the following propositions.
\begin{proposition}\label{Prop: norm upper bound}
Fix any $H>1/4$. There exists a positive constant $C$ such that
$$\tilde{d}(g)\ \text{and} \ \tilde{d}_R(g)\leq C,$$
for all $g\in \hG_N(\mr^d)$ with $\|g\|_{\textsc{CC}}=1$.
\end{proposition}
\begin{proof}
When $H\leq 1/2$, the claimed result follows from Lemma \ref{lem: continuous embedding when H<1/2} and the definition of $\|\cdot\|_{\textsc{CC}}$. In what follows, we prove our result for $H>1/2$.

Observe that, by the definition of $\tilde{d}$ and $\tilde{d}_R$,
\begin{align}\label{d<dR}\tilde{d}(g)\leq \tilde{d}_R(g),\quad \text{for\ all}\ g\in \hG_N(\mr^d).\end{align}
We thus only need to prove the claimed upper bound for $d_R$. 

Pick any $g\in \hG_N(\mr^d)$ with $\|g\|_{\textsc{CC}}=~1$. 
Thanks to Remark \ref{connecting by nondegenerate path}, we can find an $h\in\bch([0,1])$ such that 
\begin{itemize}
\item[(a)] $\Phi({h})_1=g$;
\item[(b)] $\Phi_1:{\bch}([0,1])\to \hG_N(\mr^d)$ is non-degenerate at ${h}$, that is, $\langle \D\Phi(h),\D\Phi(h)\rangle_{\cH}$ is non-degenerate; and
\item[(c)] $\|{h}\|_{{\bch}([0,1])}\leq C,$ for some constant $C$ not depending on $u$. 
\end{itemize}
The existence of such an $h$ with the above three properties implies immediately that 
$$\tilde{d}_R(g)\leq C,\quad\text{for\ all}\ g\in \hG_N(\mr^d)\ \text{with}\ \|g\|_{\textsc{CC}}=1.$$
The proof is thus completed.
\end{proof}

\begin{proposition}\label{Prop: norm lower bound}
Fix any $H>1/4$. There exists a positive constant $c$ such that
$$\tilde{d}(g)\ \text{and}\ \tilde{d}_R(g) \geq c,$$
for all $g\in \hG_N(\mr^d)$ with $\|g\|_{\textsc{CC}}=1$.
\end{proposition}
\begin{proof}
Thanks again to relation \eqref{d<dR}, we only need to prove the claimed lower bound for $\tilde{d}$.

Note that $C_0^H([0,1]; \mr^d)$ is continuously embedded into $C^{1/H-\rm{var}}([0,T],\mr^d)$. By Lemma \ref{lemma: variational embedding}, we therefore have the following continuous embedding,
\begin{align}\label{embedding}\bar{\cH}\hookrightarrow C^{q-\rm{var}}([0,T];\mr^d),\end{align}
for $q=1/H$ when $H>1/2$; and $q>(H+1/2)^{-1}$ when $H\leq 1/2$. Observe that in either case, we can pick $q<2$.

In what follows, we prove our proposition by contradiction. Suppose the claimed result is not true for $\tilde{d}$. There exists a sequence $g_n\in G_N(\mr^d)$ with $\|g_n\|_{\textsc{CC}}=1$ such that 
$$\tilde{d}(g_n)\downarrow0,\quad \text{as}\ \,n\to\infty.$$
By the definition of $\tilde{d}$, we can find a sequence  $h_n\in{\bch([0,1])}$ such that $S_N(h_n)_1=g_n$ and 
$$\|h_n\|_{\bch([0,1])}\downarrow0,\quad\text{as}\ \,n\to\infty.$$
According to the continuous embedding in \eqref{embedding}, we must have for some $q<2$,
$$\|h_n\|_{q-\rm{var};[0,1]}\downarrow0\quad\text{as}\ \,n\to\infty.$$
This, together with the standard estimate for Young's integral, implies
$$g_n=S_N(h_n)_1\to \mathbf{1},$$
where $\mathbf{1}$ is the group identity in $\hG_N(\mr^d)$, and the convergence takes place in the topology induced by $\|\cdot\|_{\textsc{CC}}$. This contradicts with our assumption that $\|g_n\|_{\textsc{CC}}=1$. The proof is thus completed.
\end{proof}

Finally, we are ready to state and prove our main theorem in this section.
\begin{theorem}\label{th: equivalence of distances G}
Fix any $H>1/4$. There exit positive constants $c$ and $C$, not depending on $g$, such that 
$$c\,\|g\|_{\textsc{CC}}\leq \tilde{d}(g)\leq \tilde{d}_R(g)\leq C\,\|g\|_{\textsc{CC}},$$
for all $g\in \hG_N(\mr^d)$.  Therefore, $\tilde{d}(\cdot)$, $\tilde{d}_R(\cdot)$, and $\|\cdot\|_{\textsc{CC}}$ are equivalent.
\end{theorem}
\begin{proof}
The second inequality is simply \eqref{d<dR}. Hence we only need to prove that $\tilde{d}(\cdot)$ (respectively, $\tilde{d}_R(\cdot)$) and $\|\cdot\|_{\textsc{CC}}$ are equivalent.

 Recall that $\triangle$ is the canonical dilation on $\hG_N(\mr^d)$. For any $h\in\bch([0,1])$ and $\lambda\in\mr$ we have
$$S_N(\lambda h)=\triangle_\lambda S_N(h).$$
Therefore both $\tilde{d}$ and $\tilde{d}_R$ are homogeneous with respect to the dilation, that is,
$$\tilde{d}(\triangle_\lambda g)=|\lambda|\tilde{d}(g)\ \ \text{and}\ \ \tilde{d}_R(\triangle_\lambda g)=|\lambda|\tilde{d}_R(g),\ \quad \text{for\ all}\ g\in \hG_N(\mr^d), \lambda\in\mr.$$
The rest of the proof then follows from Proposition \ref{Prop: norm upper bound} and Proposition \ref{Prop: norm lower bound}.
\end{proof}

As a direct corollary  of Remark \ref{rem: identify distances g and G} and Theorem \ref{th: equivalence of distances G}, we have the following equivalence of ``distances" on $\mathbb{R}^n$.

\begin{corollary}\label{th: equivalence of distances g}
Fix any $H>1/4$. There exit positive constants $c$ and $C$, not depending on $u$, such that 
$$c\,\|u\|_{\textsc{CC}}\leq {d}(u)\leq {d}_R(u)\leq C\,\|u\|_{\textsc{CC}},$$
for all $u\in \mathbb{R}^n$.  Therefore, ${d}(\cdot)$, ${d}_R(\cdot)$, and $\|\cdot\|_{\textsc{CC}}$ are equivalent.
\end{corollary}

%
%
%

%
%

%
%
%

\end{document}